\documentclass[reqno,12pt]{amsart}

\numberwithin{equation}{section}

\usepackage{verbatim} % hacer comentarios
\usepackage{xspace} % dejar espacios
\usepackage{amssymb}
\usepackage{amsmath}
\usepackage{amsthm} %
\usepackage[latin1]{inputenc}
\usepackage{graphicx}
\usepackage{amscd,amssymb,euscript,graphics}
\makeindex
\newcommand{\R}{\ensuremath{\mathbb{R}}}
\newcommand{\C}{\ensuremath{\mathbb{C}}}

\newcommand{\N}{\ensuremath{\mathbb{N}}}

\newcommand{\D}{\ensuremath{\mathbb{D}}}
\newcommand{\norm}[1]{\ensuremath{\|#1\|}}

\newtheorem{teo}{Theorem}[section]
\newtheorem{prop}[teo]{Proposition}
\newtheorem{lem}[teo]{Lemma}
\newtheorem{cor}[teo]{Corollary}

\newtheorem{ej}[teo]{Example}
\newtheorem{obs}[teo]{Remark}
\newtheorem{question}{Question}
\newtheorem{notation}[teo]{Notation}
\newcount\theTime
\newcount\theHour
\newcount\theMinute
\newcount\theMinuteTens
\newcount\theScratch
\theTime=\number\time
\theHour=\theTime
\divide\theHour by 60
\theScratch=\theHour
\multiply\theScratch by 60
\theMinute=\theTime
\advance\theMinute by -\theScratch
\theMinuteTens=\theMinute
\divide\theMinuteTens by 10
\theScratch=\theMinuteTens
\multiply\theScratch by 10
\advance\theMinute by -\theScratch

\def\today{{\number\day\space
 \ifcase\month\or
  January\or February\or March\or April\or May\or June\or
  July\or August\or September\or October\or November\or December\fi
 \space\number\year}}

\begin{document}

\author{GABRIEL H. TUCCI}
\title[Quasinilpotent generators]{Some quasinilpotent generators of the hyperfinite $\mathrm{II}_1$ factor}
\address{Department of Mathematics, Texas A\&M University, College Station, TX 77843-3368, USA}
\email{gtucci@math.tamu.edu}

%\date{\today}

\begin{abstract}
For each sequence $\{c_n\}_n$ in $l_{1}(\N)$ we define an operator $A$ in the hyperfinite $\mathrm{II}_1$-factor $\mathcal{R}$. We prove that these operators are quasinilpotent and they generate the whole hyperfinite $\mathrm{II}_1$-factor. We show that they have non-trivial, closed, invariant subspaces affiliated to the von Neumann algebra and we provide enough evidence to suggest that these operators are interesting for the hyperinvariant subspace problem. We also present some of their properties. In particular, we show that the real and imaginary part of $A$ are equally distributed, and we find a combinatorial formula as well as an analytical way to compute their moments. We present a combinatorial way of computing the moments of $A^{*}A$.
\end{abstract}

\maketitle

\vspace{0.3cm}

\section{Introduction}

\par Consider a von Neumann algebra $\mathcal{M}$ acting on a Hilbert space $\mathcal{H}$. A closed subspace $\mathcal{H}_0$ of $\mathcal{H}$ is said to be affiliated with $\mathcal{M}$ if the projection of $\mathcal{H}$ onto $\mathcal{H}_0$ belongs to $\mathcal{M}$. The subspace $\mathcal{H}_0$ is said to be non-trivial if $\mathcal{H}_0\neq 0$ and $\mathcal{H}_0\neq \mathcal{H}$. For $T\in \mathcal{M}$, a subspace $\mathcal{H}_0$ is said to be $T$-invariant, if $T(\mathcal{H}_0)\subseteq \mathcal{H}_0$, i.e. if $T$ and the projection $P_{\mathcal{H}_0}$ onto $\mathcal{H}_0$ satisfy
$$P_{\mathcal{H}_0}TP_{\mathcal{H}_0}=TP_{\mathcal{H}_0}.$$
$\mathcal{H}_0$ is said to be hyperinvariant for $T$ (or $T$-hyperinvariant) if it is $S$-invariant for every $S\in \mathcal{B}(\mathcal{H})$ that commutes with $T$. If the subspace $\mathcal{H}_0$ is $T$-hyperinvariant, then $P_{\mathcal{H}_0}\in W^{*}(T)=\{T,T^{*}\}''$ (cf. \cite{B-circ}). However, the converse statement does not hold true. In fact, one can find $A\in M_{3}(\C)$ and an $A$-invariant projection $P\in W^{*}(A)$ which is not $A$-hyperinvariant (cf. \cite{B-circ}).
\par The invariant subspace problem relative to the von Neumann algebra $\mathcal{M}$ asks whether every operator $T$ has a non-trivial, closed, invariant subspace $\mathcal{H}_0$ affiliated with $\mathcal{M}$, and the hyperinvariant subspace problem asks whether one can always choose such an $\mathcal{H}_0$ to be hyperinvariant for $T$. Of course, if $\mathcal{M}$ is not a factor, then the answer to both of these questions is {\it yes}. Also, if $\mathcal{M}$ of finite dimension, i.e. $\mathcal{M}\cong M_{n}(\C)$ for some $n\in\N$, then every operator in $\mathcal{M}\backslash \C 1$ has a non-trivial eigenspace, and therefore a non-trivial $T$-invariant subspace. Recall from \cite{Brown} that every operator in a $\mathrm{II}_1$-factor defines a probability measure $\mu_{T}$ on $\C$, the Brown measure of $T$, with $\mathrm{supp}(T)\subseteq \sigma(T)$. In \cite{Ha}, Uffe Haagerup and Hanne Schultz made a huge advance in this problem. Namely, they proved that if the Brown measure of the operator $T$ is not concentrated in one point, then the operator $T$ has a non-trivial, closed, invariant subspace, affiliated with $\mathcal{M}$ and moreover, this subspace is hyperinvariant. More specifically, for each Borel set $B\subseteq \C$, they constructed a maximal, closed, $T$-invariant subspace, $\mathcal{K}=\mathcal{K}_T(B)$, affiliated with $\mathcal{M}$, such that the Brown measure of $T\vert_{\mathcal{K}}$ is concentrated on $B$ and if we denote by $P$ the projection onto this subspace, then $\tau(P)=\mu_{T}(B)$. Therefore, if $\mu_{T}$ is not a Dirac measure, then $T$ has a non-trivial invariant subspace affiliated with $\mathcal{M}$. If the Borel set $B$ is a closed ball of radius $r$ centered at $\lambda$. Then $\mathcal{K}_T(B)$ is the set of vectors $\xi\in\mathcal{H}$, for which there is a sequence $\{\xi_n\}_{n}$ in $\mathcal{H}$ such that
$$\lim_{n}{\norm{\xi_n-\xi}}=0 \textrm{\quad and \quad} \limsup_{n}{\norm{(T-\lambda 1)^{n}\xi_n}^{\frac{1}{n}}}\leq r.$$

\par As regards the invariant subspace problem relative to the von Neumann algebra, the following question remains completely open: If $T$ is an operator in a $\mathrm{II}_1$-factor $\mathcal{M}$ and if the Brown measure $\mu_{T}$ is a Dirac measure, for example if $T$ is quasinilpotent, does $T$ has a non-trivial closed, invariant subspace affiliated with $W^{*}(T)$?   

\vspace{0.3cm}

\par In \cite{DH}, Dykema and Haagerup introduced the family of DT-operators and they studied many of their properties. The case of the quasinilpotent DT-operator arose as a natural candidate for an operator without an invariant subspace affiliated to the von Neumann algebra. Later on, in \cite{DT1}, Dykema and Haagerup finally showed that every quasinilpotent DT-operator $T$ has a one-parameter family of non-trivial hyperinvariant subspaces. In particular, they proved that for $t\in [0,1]$, 
$$\mathcal{H}_t:=\Big \{\xi\in\mathcal{H} \,:\, \limsup_{n}{\Big( \frac{k}{e}\norm{T^{k}\xi}\Big )^{\frac{2}{k}}}\leq t \Big\}$$
is a closed, hyperinvariant subspace of $T$.

\vspace{0.5cm}

\par In this paper, for each sequence $\{c_n\}_n\in l_1(\N)$ we define an operator $A$ in the hyperfinite $\mathrm{II}_1$-factor. These operators are quasinilpotent, and under a certain mild restriction on the sequence $\{c_n\}_n$ they generates the whole hyperfinite $\mathrm{II}_1$-factor. As a corollary of the proof that $A$ is quasinilpotent we deduce that given $\{c_n\}_n\in l_1(\N)$ then 
$$\limsup_{k}{(k!\,\,\sigma_k)^{1/k}}=0 \textrm{\quad where\quad } \sigma_k:=\sum_{1\leq n_1<n_2<\ldots <n_k}{|c_{n_1}c_{n_2}\ldots c_{n_k}|}.$$
We also show that these operators have invariant subspaces affiliated with the von Neumann algebra. The projections onto these subspaces live in the diagonal masa 
$$\mathcal{D}:=\overline{\Bigg(\bigotimes_{n=1}^{+\infty}{\D_{2}(\C)}\Bigg)}^{\mathrm{WOT}}\subset \mathcal{R}$$
where $\D_{2}(\C)$ is the algebra of the $2\times 2$ diagonal matrices. We also show that none of these projections is hyperinvariant. Moreover, we show that if $p$ is a non-trivial hyperinvariant projection for $A$ then 
$$p\notin \bigcup_{n=1}^{+\infty}\Big({\bigotimes_{k=1}^{n}{M_2(\C)}}\Big).$$

\vspace{0.4cm}
\noindent In section $\S4$ we show that these operators have trivial kernel and dense range. We prove also that given $r>0$ and any sequence $\{\gamma_n\}_{n=1}^{+\infty}$ of positive numbers, if we define the subspace $\mathcal{H}_{r}(A)$ by

\begin{equation*}
\mathcal{E}_{r}(A):=\{\xi\in\mathcal{H}\,:\, \limsup_{n}{\,\gamma_n\norm{A^{n}(\xi)}^{1/n}}\leq r\}\quad \textrm{and}\quad  \mathcal{H}_{r}(A)=\overline{\mathcal{E}_{r}(A)}
\end{equation*}

\noindent then this subspace is either $\mathcal{H}$ or $\{0\}$. We are unable to determine if the operator $A$ has a non-trivial hyperinvariant subspace, and for the evidence showed above, it is a possible counterexample to the hyperinvariant subspace problem.

\vspace{0.4cm}

\noindent In section $\S5$, we show that the real and imaginary part of $A$, $a:=\mathrm{Re}(A)$ and $b:=\mathrm{Im}(A)$, are equally distributed. We find a combinatorial formula as well as an analytical way to compute their moments. We also compute some of their mixed moments. We prove also that when $c_{n}=\alpha^{n}$ where $0<\alpha\leq \frac{1}{2}$ then $W^{*}(a)$ is a Cartan masa in the hyperfinite and we find countably many values of $\alpha\in (\frac{1}{2},1)$ in which $W^{*}(a)$ is not maximal abelian. However, for all the values of $\alpha\in (0,1)$ this algebra is diffuse.
In section $\S6$, we find a combinatorial formula for the moments of $A^{*}A$ in terms of alternating partitions of elements of two different colors. We also ask a question regarding these partitions.

\vspace{0.3cm}
\noindent{\it Acknowledgment:} I thank my advisor, Ken Dykema, for many helpful discussions and comments.

\vspace{0.6cm}

\section{Notation and Preliminaries}

\subsection{Infinite tensor products of finite von Neumann algebras}
\par The Hilbert space tensor product of two Hilbert spaces is the completion of their algebraic tensor product. One can define a tensor product of von Neumann algebras (a completion of the algebraic tensor product of the algebras considered as rings), which is again a von Neumann algebra, and acts on the tensor product of the corresponding Hilbert spaces. The tensor product of two finite algebras is finite, and the tensor product of an infinite algebra and a non-zero algebra is infinite. The type of the tensor product of two von Neumann algebras (I, II, or III) is the maximum of their types. The Tomita commutation Theorem for tensor products states that
$$(M\,\overline{\otimes}\, N)'=M'\,\overline{\otimes}\, N'.$$
\par The tensor product of an infinite number of von Neumann algebras, if done naively, is usually a ridiculously large non-separable algebra. Instead one usually chooses a state on each of the von Neumann algebras, uses this to define a state on the algebraic tensor product, which can be used to product a Hilbert space and a (reasonably small) von Neumann algebra. Given finite factors $\{\mathcal{M}_n \}_{n=1}^{+\infty}$, denote $\tau_{n}$ the unique faithful normal trace on $\mathcal{M}_n$. We write $\bigotimes_{n=1}^{+\infty}{\mathcal{M}_n}$ for the algebraic tensor product, that is finite linear combination of elementary tensors $\bigotimes_{n=1}^{+\infty}{x_n}$, where $x_n\in\mathcal{M}_n$ and all but finitely many $x_n$ are 1. We have the product state $\tau$ on $\bigotimes_{n=1}^{+\infty}{\mathcal{M}_n }$ defined on elementary tensors by 

$$\tau \Bigg(\bigotimes_{n=1}^{+\infty}{x_n}\Bigg)=\prod_{n=1}^{+\infty}{\tau_{n}(x_n)}.$$

\vspace{0.2cm}

\noindent Now let $\pi$ be the representation of $\bigotimes_{n=1}^{+\infty}{\mathcal{M}_n}$ by left multiplication on the Hilbert space $L^2\Big(\bigotimes_{n=1}^{+\infty}{\mathcal{M}_n}\Big)$ in the usual way. The infinite von Neumann tensor product of the $\mathcal{M}_n$ is then the weak-closure of the image of $\pi$. This is necessarily a finite factor, as it has a trace, namely the extension of $\tau$, which is the unique normalized trace. The Tomita commutation Theorem remains true in this infinite setting.

\vspace{0.4cm}

\subsection{The hyperfinite $\mathrm{II}_1$-factor}
\par Let $\mathcal{M}$ a finite von Neumann algebra and $\tau$ a faithful normal trace. Given an element $x$ in such a von Neumann algebra, we will denote $\norm{x}_2=\tau(x^{*}x)^{1/2}$. Let $L^{2}(\mathcal{M})$ the Hilbert space obtained by the completion of $\mathcal{M}$ with respect to the $\norm{\cdot}_2$. We shall follow the tradition in the subject of regarding $\mathcal{M}$ as a subset of $L^{2}(\mathcal{M})$ whenever it is convenient. The standard form is the representation of $\mathcal{M}\subset \mathcal{B}(L^{2}(\mathcal{M}))$ obtained by letting each $x$ in $\mathcal{M}$ act by left multiplication on $L^{2}(\mathcal{M})$.
\par Murray and von Neumann defined the approximate finite dimensional property (AFD). Namely, a $\mathrm{II}_1$-factor $\mathcal{M}$ is said to be AFD when for any $x_1,\ldots,x_n\in \mathcal{M}$ and strong neighborhood $V$ of $0$ in $\mathcal{M}$ there exists a finite dimensional $*$-subalgebra $\mathcal{N}$ of $\mathcal{M}$ such that $x_{i}\in \mathcal{N}+V$ for each $i$. Let $M_{2}(\C)$ be the algebra of $2\times 2$ matrices. Then the infinite tensor product 
\begin{equation}\label{R}
\mathcal{R}:=\overline{\Bigg(\bigotimes_{n=1}^{+\infty}{M_{2}(\C)}\Bigg )}^{\mathrm{WOT}} 
\end{equation}
produced with respect to the unique normalized trace on $M_{2}(\C)$ is a $\mathrm{II}_1$-factor, which is obviously AFD. In \cite{vN}, Murray and von Neumann showed that up to isomorphism this is the unique AFD $\mathrm{II}_1$-factor. In complete contrast with the $C^{*}$-case, the resulting object is independent of the size of the matrices algebras involved.
\par Given a discrete group we can always define a finite von Neumann algebra via the left or right regular representation. This algebra is called the group von Neumann algebra. The group von Neumann algebra of a discrete group with the infinite conjugacy class property is a factor of type $\mathrm{II}_1$, and if the group is amenable and countable then the factor is AFD. There are many groups with these properties, as any group such that any finite subset generates a finite subgroup is amenable. For example, the group von Neumann algebra of the infinite symmetric group of all permutations of a countable infinite set that fix all but a finite number of elements is the hyperfinite type $\mathrm{II}_1$ factor.

\vspace{0.6cm}

\section{Quasinilpotent Generators}
\noindent In this section we construct the operators described before. We define the $2\times 2$ matrices $V$, $Q$ and $P$ by
\begin{equation*}
V:=\left( \begin{array}{cc}
0 & 1\\
0 & 0\\
\end{array} \right),\quad 
Q:=V^{*}V=\left( \begin{array}{cc}
0 & 0\\
0 & 1\\
\end{array} \right), \quad
P:=VV^{*}=\left( \begin{array}{cc}
1 & 0\\
0 & 0\\
\end{array} \right). 
\end{equation*}

\noindent Let $\{c_{n}\}_n$ be a sequence in $l_{1}(\N)$ and let us consider $A_n:=c_1\, V+c_2\, I\otimes V +\ldots +c_n\, I^{\otimes (n-1)}\otimes V\in\mathcal{R}$ then $\norm{A_n}\leq \sum_{k=1}^{n}{|c_k|}$ for all 
$n\geq 1$. The sequence $\{A_n\}_n$ is Cauchy in norm since by assumption $\sum_{n=1}^{+\infty}{|c_n|}<+\infty$. 
Therefore, it converges in the operator norm to an operator $A$ in the hyperfinite $\mathrm{II}_1$-factor with
\begin{equation}\label{defA}
A:=\sum_{n=1}^{+\infty}{c_n \,V_{n}} \quad\quad\textrm{where}\quad\quad V_{n}:=I^{\otimes (n-1)}\otimes V. 
\end{equation}
We will prove that this operator is quasinilpotent and that under certain mild hypothesis it generates the hyperfinite 
$\mathrm{II}_1$ factor $\mathcal{R}$.

\vspace{0.4cm}

\begin{teo}
The operator $A$ in (\ref{defA}) is quasinilpotent. 
\end{teo}

\begin{proof}
Let $A=\sum_{n=1}^{+\infty}{c_n\, I^{\otimes (n-1)}\otimes V}$ then using that $V^{2}=0$ we see that
$$A^{k}=k!\sum_{1\leq n_1<n_2<\ldots <n_k}{c_{n_1}c_{n_2}\ldots c_{n_k}\,\, V_{n_1}V_{n_2}\ldots V_{n_k}} \textrm{\quad where \quad} V_{n}:=I^{\otimes(n-1)}\otimes V.$$ 
Then $\norm{A^{k}}\leq k!\sum_{1\leq n_1<n_2<\ldots <n_k}{|c_{n_1}c_{n_2}\ldots c_{n_k}|}$. Let us define 
$$\sigma_k:=\sum_{1\leq n_1<n_2<\ldots <n_k}{|c_{n_1}c_{n_2}\ldots c_{n_k}|} \,\,\,\,\, \mathrm{and}\,\,\,\,\, 
\sigma_{0}:=1.$$ 
Therefore,
\begin{equation}\label{Ak}
\norm{A^{k}}\leq k! \sigma_k.
\end{equation}

\vspace{0.3cm}
\noindent Consider the function $f(z):=\Pi_{n=1}^{+\infty}{(1+zc_n)}=\sum_{n=0}^{+\infty}{\sigma_n z^n}.$ 
Using the Weierstrass factorization Theorem \cite{Gr} and the fact that the sequence $\{c_n\}_n$ is absolutely sumable we see that 
the function $f(z)$ is entire. From this function $f(z)$ we define formally $g(z)$ by

\begin{equation}
g(z):=\int_{0}^{+\infty}{f(tz)e^{-t}\,dt}.
\end{equation}

\noindent Now we will prove that the function $g(z)$ is well defined, entire and its power series expansion is
$g(z)=\sum_{n}{n!\sigma_n z^n}$. Therefore, $\limsup{(n!\sigma_n)^{1/n}}=0$ and using (\ref{Ak}), we deduce that $A$ is quasinilpotent.\\

\noindent Take $R>0$ then there exists $N_0$ such that $\sum_{k=N_0+1}^{+\infty}{|c_k|}<\frac{1}{2R}$. Then for $|w|<R$

$$|f(tw)|=\Pi_{n=1}^{+\infty}{|1+wtc_n|}=\Pi_{n=1}^{N_0}{|1+wtc_n|}\,\cdot\, \Pi_{n=N_0+1}^{+\infty}{|1+twc_n|}\leq $$
$$\leq K_R \,\cdot\, \mathrm{exp}\Bigg(\sum_{n=N_0+1}^{+\infty}{|c_ntw|}\Bigg)=K_R \cdot 
\mathrm{exp}\Bigg(t|w|\sum_{n=N_0+1}^{+\infty}{|c_n|}\Bigg)\leq$$
$$\leq K_R\cdot\mathrm{exp}\Bigg(t|w|\frac{1}{2R}\Bigg)\leq K_R\cdot \mathrm{exp}(t/2).$$

\vspace{0.3cm}
\noindent Therefore, 
\begin{equation}\label{eqg}
|g(w)|\leq \int_{0}^{+\infty}{K_R\,\cdot\, e^{-t/2}\,dt}=\frac{K_R}{2}, \,\,\,\,\,\mathrm{for\,\,all}\,\,\,\,\,\, |w|<R.
\end{equation}

\noindent Then $g$ is 
a well defined function for all $w\in \C$. Moreover, for all closed curves $\gamma$ contained in the disk of radius $R$
centered at origin we have

$$\oint_{\gamma}{g(z)dz}=\oint_{\gamma}{\Bigg(\int_{0}^{+\infty}{f(tz)e^{-t}\,dt}\Bigg)dz}=
\int_{0}^{+\infty}{e^{-t}\Bigg(\oint_{\gamma}{f(tz)dz}\Bigg)dt}=0.$$
\noindent(Note that we are allowed to interchange the integrals by applying Fubini's Theorem since $g$ is bounded (\ref{eqg}).) Using Morera's Theorem we see that $g$ is holomorphic in the disk of radius $R$, and since $R$ is arbitrary, $g$ is entire. The fact that the function $g$ has the desired power series expansion comes from the fact that
$k!=\int_{0}^{+\infty}{t^ke^{-t}dt}$.
\end{proof}

\vspace{0.4cm}
From the proof of the last Theorem we observe that something a little bit more general was proved. We state it in the next corollary.

\vspace{0.3cm}
\begin{cor} Let $\{c_n\}_n$ a sequence of complex number in $l_1(\N)$. Then $$\limsup_{k}{(k!\,\,\sigma_k)^{1/k}}=0$$ where 
$\sigma_k:=\sum_{1\leq n_1<n_2<\ldots <n_k}{|c_{n_1}c_{n_2}\ldots c_{n_k}|}$.
\end{cor}

\vspace{0.6cm}

In the next Theorem we will prove that under certain mild hypothesis the operator $A$ generates the whole hyperfinite
$\mathrm{II}_1$-factor $\mathcal{R}$ as in (\ref{R}).

\vspace{0.3cm}
\begin{teo}
Let $\{c_n\}_n$ be a sequence of complex numbers in $l_1(\N)$ such that $|c_{i}|\neq |c_{j}|$ whenever $i\neq j$ and $c_{j}\neq 0$ for all $j\geq 1$.
Then the von Neumann algebra generated by $A$ is $\mathcal{R}$. Moreover, if there exist $i\neq j$ so that $|c_{i}|=|c_{j}|$ then the von Neumann algebra generated by $A$ is not the whole hyperfinite factor.
\end{teo}

\vspace{0.3cm}

\begin{proof}
By applying an automorphism, if necessary, we can assume without loss of generality that $c_1>c_2>\ldots>c_n>c_{n+1}>\dots>0$. Let us define 
\begin{equation}
q_k:=\sum_{n=1}^{+\infty}{c_n^k\,I^{\otimes(n-1)}\otimes Q}\,\,\,\,\,\,\,\,\,\,
p_k:=\sum_{n=1}^{+\infty}{c_n^k\,I^{\otimes(n-1)}\otimes P}
\end{equation}
and
\begin{equation*}
v_{n,m}:=I^{\otimes(n-1)}\otimes V\otimes I^{\otimes(m-n-1)}\otimes V^{*}+I^{\otimes(n-1)}\otimes V^*\otimes 
I^{\otimes(m-n-1)}\otimes V \,\,\,\,\mathrm{for}\,\,\,n<m.
\end{equation*}

\vspace{0.3cm}
\noindent Then we can see that $A^{*}A=q_2+\sum_{1\leq n <m}{c_nc_m\,v_{n,m}}=q_2+v$ and 
$AA^{*}=p_2+\sum_{1\leq n <m}{c_nc_m\,v_{n,m}}=p_2+v$ where $v:=\sum_{1\leq n <m}{c_nc_m\,v_{n,m}}$. Observing that 
$p_2+q_2=\sum_{n=1}^{+\infty}{c_{n}^2}$ we see that $\{v,p_2,q_2\}\in W^{*}(A)$.\\
\noindent Then $q_2A-Aq_2=\sum_{n=1}^{+\infty}{c_n^3 \,I^{\otimes(n-1)}\otimes V}$ and therefore $q_6\in W^{*}(A)$. 
Repeating the same argument we see that given $k\geq 1$ there exists $N(k)\geq k$ such that $q_{N(k)}\in W^{*}(A)$. 
Now observing that

\begin{equation*}
\lim_{k}{\Bigg(\sum_{n=1}^{+\infty}{c_n^{N(k)}\Bigg)^{\frac{1}{N(k)}}}}
=\max \{c_n\,:\,n\geq1 \}=c_1>c_2=\lim_{k}{\Bigg(\sum_{n=2}^{+\infty}{c_n^{N(k)}\Bigg)^{\frac{1}{N(k)}}}}
\end{equation*}
\vspace{0.2cm}

\noindent we obtain $Q$ as a spectral projection of $q_{N(k)}$ for $k$ sufficiently large. So, $Q\in W^{*}(A)$ and since 
$AQ-QA=c_1\,V$ we have also that $V\in W^{*}(A)$. Repeating the same 
argument now using that $c_2>c_3$ we obtain that $I\otimes Q,\,\,I\otimes V\in W^{*}(A)$. Analogously, 
$I^{\otimes (n-1)}\otimes Q,\,\,I^{\otimes (n-1)}\otimes V\in W^*(A)$ for all $n\geq 1$. Then 
$$\{I^{\otimes (n-1)}\otimes Q,\,\,I^{\otimes (n-1)}\otimes P,\,\,I^{\otimes (n-1)}\otimes V,\,\,
I^{\otimes (n-1)}\otimes V^{*}\,\,:\,\, n\geq 1\}\in W^{*}(A)$$ 
and we conclude that $W^{*}(A)=\mathcal{R}$.

\vspace{0.3cm}
\noindent If there exist $n\neq m$ such that $|c_{n}|=|c_{m}|$, then by applying an automorphism as we did before we can assume that $c_{n}=c_{m}$. It is a direct computation to check that the operators 
$$S_{n,m}:=P_nQ_m + Q_nP_m -\frac{c_n}{c_m}V_nV^*_m-\frac{c_m}{c_n}V^*_nV_m$$
commute with $A$. Note that if $c_{n}=c_{m}$ this operator is selfadjoint and commutes with $A$ and hence the von Neumann algebra generated by $A$ is not the whole hyperfinite.
\end{proof}

\vspace{0.6cm}

\noindent The operator $A_n:=c_1\, V+c_2\, I\otimes V +\ldots +c_n\, I^{\otimes (n-1)}\otimes V$ is a nilpotent operator of 
order $n+1$ and $A_n^{n}:=n!\,c_1c_2\ldots c_n\, V\otimes V\otimes \ldots \otimes V$. Since the projection $P^{\otimes(n)}$
is the orthogonal projection onto the range of $A_n^n$ it is an hyperinvariant projection for the operator $A_n$. Since $A$
commutes with $A_n$ it is an invariant projection for $A$ affiliated to the von Neumann algebra $\mathcal{R}$. However, we will see
that none of these invariant projections for $A$ are $A$--hyperinvariant.\\
\noindent Given $1\leq n$ we will denote by $V_n:=I^{\otimes(n-1)}\otimes V$ and analogously with $Q_n,\,P_n$ and $V^*_n$. Let $n<m$ and consider the operator $S_{n,m}$ defined by
\begin{equation}\label{S}
S_{n,m}:=P_nQ_m + Q_nP_m -\frac{c_n}{c_m}V_nV^*_m-\frac{c_m}{c_n}V^*_nV_m.
\end{equation}
As we mention before, $AS_{n,m}=S_{n,m}A$ for all $1\leq n<m$ and we can see that
$$P^{\otimes(n)}S_{n,n+1}P^{\otimes(n)}=P^{\otimes(n)}\otimes Q \,\,\,\,\mathrm{and}\,\,\,\, S_{n,n+1}P^{\otimes(n)}=P^{\otimes(n)}\otimes Q-\frac{c_{n+1}}{c_{n}}P^{\otimes(n-1)}\otimes V^*\otimes V.$$ 
So the projection $P^{\otimes(n)}$ is not invariant for $S_{n,n+1}$ and therefore, not $A$--hyperinvariant for $n\geq 1$.

\vspace{0.7cm}

\noindent Denote by $\D_{2}(\C)$ the algebra of the $2\times 2$ diagonal matrices. Then 
$$\mathcal{D}:=\overline{\Bigg(\bigotimes_{n=1}^{+\infty}{\D_{2}(\C)}\Bigg)}^{\mathrm{WOT}}\subset \mathcal{R}$$
is a maximal abelian subalgebra (masa) of $\mathcal{R}$. Therefore, $\mathcal{D}\cong L^{\infty}[0,1]$ and under this identification the projection $P$ corresponds to the characteristic function on $[0,1/2]$ and the projection $Q$ to the characteristic function on $[1/2,1]$ and so on.
\par Given a word with letters in the alphabet $\{P,Q,V,V^*\}$ we can associate an element in $\mathcal{R}$ by adding a tensor product between each of the letters. For example, the word $VPV^*Q$ corresponds to the element $V\otimes P\otimes V^{*}\otimes Q$ and so on. Note that if the word consists only of letters $P$ and $Q$ the associated element is a projection in the diagonal algebra $\mathcal{D}$, and under the identification with $L^{\infty}[0,1]$, the words in $P$ and $Q$ correspond to dyadic intervals in $[0,1]$.\\

\noindent Now we will prove the following Proposition.

\vspace{0.3cm}

\begin{prop}\label{SPw}
Given any word $w$ with letters in $\{P,Q\}$ the corresponding projection $p_w\in \mathcal{D}$ is not $A$-hyperinvariant. Moreover,
$$\bigvee_{S\in\{A\}'\cap \mathcal{R}}{\overline{\mathrm{Range}(Sp_w)}}=\mathcal{H}.$$
\end{prop}

\begin{proof}
Let's first consider the case $w=P$. Let $S_{1,n}$ be the operator defined in (\ref{S}). Then 
$$S_{1,n}P=P\otimes I^{\otimes (n-2)}\otimes Q -\frac{c_n}{c_1}\,V^{*}\otimes I^{\otimes (n-2)}\otimes V.$$
Hence for $n\geq 2$
$$\mathrm{Range}(P)\vee \mathrm{Range}(S_{1,n}P)=\mathrm{Range}(P+Q\otimes I^{\otimes (n-2)}\otimes P).$$

\noindent Since 
\begin{equation}\label{vee}
\bigvee_{n\geq 2}\mathrm{Range}(Q\otimes I^{\otimes (n-2)}\otimes P)=\mathrm{Range}(Q)
\end{equation}
we see that
$$\bigvee_{S\in\{A\}'\cap \mathcal{R}}{\overline{\mathrm{Range}(SP)}}=\mathcal{H}.$$

\vspace{0.3cm}

\noindent The case $w=Q$ follows similarly. Reasoning by induction in the length of the word let's assume that it is true for all the words of length $n$. Take any word $v$ of length $n+1$. Without loss of generality we can assume that it ends with $Q$ (the other case follows similarly). Then $v=w\otimes Q$ where $w$ is a word of length $n$. Thus for $m\geq n+2$
$$S_{n+1,m}(w\otimes Q)=w\otimes Q \otimes I^{\otimes (m-n-2)}\otimes P-\frac{c_{n+1}}{c_m}\,w\otimes V \otimes I^{\otimes (m-n-2)}\otimes V^{*} $$
hence
$$\mathrm{Range}(w\otimes Q)\vee \mathrm{Range}(S_{n+1,m}(w\otimes Q))=\mathrm{Range}(w\otimes Q + w\otimes P\otimes I^{\otimes (m-n-2)}\otimes Q).$$
Using (\ref{vee}) again, and the induction hypothesis we obtain
$$\bigvee_{S\in\{A\}'\cap \mathcal{R}}{\overline{\mathrm{Range}(Sp_v)}}=\mathcal{H}$$
and finishes the proof.
\end{proof}

\vspace{0.4cm}

\begin{teo}\label{A-hyper}
Let $n\in{\N}$ and $p\in \mathcal{R}$ be a non-trivial $A_n$--hyperinvariant projection. Then it is not $A$--hyperinvariant. 
\end{teo}

\vspace{0.3cm}

\noindent Before proving this Theorem let's state a well known result proved by Barraa in \cite{Ba}. This is a generalization of a result proved by Domingo Herrerro for finite dimensional Hilbert spaces in \cite{He}.

\vspace{0.5cm}

\begin{teo}[Barraa]\label{Barraa}
Every non-trivial hyperinvariant subspace $\mathcal{M}$ for a nilpotent operator $A$ of order $n$ satisfies that 
$$\overline{\mathrm{Range}(A^{n-1})}\subseteq \mathcal{M}\subseteq \mathrm{Ker}(A^{n-1}).$$
\end{teo}

\vspace{0.5cm}
\begin{proof} of Theorem \ref{A-hyper}:
Let $\mathcal{M}$ be a non-trivial hyperinvariant subspace for $A_n$. Since $A_n$ is nilpotent of order $n+1$ we have by Theorem \ref{Barraa}
that $\overline{\mathrm{Range}(A_{n})}\subseteq \mathcal{M}\subseteq \mathrm{Ker}(A_{n}).$ Since $\mathrm{Ker}(A_{n})=\mathrm{Range}(1-Q^{\otimes n})$ and $\overline{\mathrm{Range}(A_{n})}=\mathrm{Range}(P^{\otimes n})$, if we denote by $p$ the projection onto $\mathcal{M}$ we have that
$P^{\otimes n}\leq p\leq 1-Q^{\otimes n}$. Using Proposition \ref{SPw} we know that 
$$\bigvee_{S\in\{A\}'\cap \mathcal{R}}{\overline{\mathrm{Ran}(SP^{\otimes n})}}=\mathcal{H}.$$
Then, there exists $S\in\{A\}'\cap \mathcal{R}$ and $h\in P^{\otimes n}(\mathcal{H})$ such that 
$$0\neq Sp(h)=SpP^{\otimes n}(h)=SP^{\otimes n}(h)\in Q^{\otimes n}(\mathcal{H})$$ 
therefore, $pSp(h)=0$. Thus, $p$ is not $S$-invariant and then not $A$-hyperinvariant.\\

%\noindent [complete] Now we will prove that $p$ is not $A_{m}$-hyperinvariant for $m\geq n+1$. First note that the operators $I^{\otimes n}\otimes %V^{*}$ and $I^{\otimes n}\otimes V$ commutes with $A_n$, thus $p(I^{\otimes n}\otimes V^{*})p=(I^{\otimes n}\otimes V^{*})p$ and  $p(I^{\otimes %n}\otimes V)p=(I^{\otimes n}\otimes V)p$. Then, $p$ commutes with $I^{\otimes n}\otimes V$ and with $I^{\otimes n}\otimes V^{*}$ and since then %$p=q\otimes I \otimes r$ where the $I$ is at the $n$-th position.
\end{proof}

\begin{obs}\label{SA}
Note that if $S\in \bigotimes_{k=1}^{n}{M_{2}(\C)}\subset \mathcal{R}$ and $SA_{n}=A_{n}S$ then $AS=SA$.
\end{obs}

\vspace{0.3cm}

\begin{teo}
Assume $p$ is a non-trivial hyperinvariant projection for $A$. Then $p\notin \bigcup_{n=1}^{+\infty}\Big(
{\bigotimes_{k=1}^{n}{M_2(\C)}}\Big)$.
\end{teo}

\begin{proof}
Assume that there exists $n\geq 1$ such that  $p\in \bigotimes_{k=1}^{n}{M_{2}(\C)}$. Since $p$ is hyperinvariant, it is $A_{n}$-invariant. Moreover,
by Remark \ref{SA}, $p$ is invariant for all $S\in \bigotimes_{k=1}^{n}{M_{2}(\C)}$ such that $SA_{n}=A_{n}S$. Hence, $p$ is $A_n$-hyperinvariant which contradicts Theorem \ref{A-hyper}. Thus, $p\notin \bigotimes_{k=1}^{n}{M_{2}(\C)}$ for any $n$.
\end{proof}

\vspace{0.4cm}

\par It will be convenient to introduce some notation at this point. Given an operator $A\in B(\mathcal{H})$ we denote by 
$\mathcal{S}(A)$ the similarity orbit of $A$. In other words, 
$\mathcal{S}(A):=\{WAW^{-1}\,:\,\mathrm{\,where\,\,\,} W \mathrm{\,\,is\,\, invertible}\}\subset B(\mathcal{H})$. As in Chapter 2 of \cite{He} we say that two operators $A$ and $B$ are asymptotically similar if $A\in \overline{\mathcal{S}(B)}$ and $B\in \overline{\mathcal{S}(A)}$, where the closure is with respect to the operator norm. Or equivalently, iff 
$\overline{\mathcal{S}(B)}=\overline{\mathcal{S}(A)}$. Now we are ready to state the next result.

\vspace{0.4cm}

\begin{prop} Let $\{a_n\}_n$ and $\{b_n\}_n$ in $l_1(\N)$ be such that $a_n,\,b_n\neq 0$ for all $n$. Let 
$A=\sum_{n=1}^{+\infty}{a_n\,V_n}$ and $B=\sum_{n=1}^{+\infty}{b_n\,V_n}$, where $V_n=I^{\otimes(n-1)}\otimes V$. 
Then $A$ and $B$ are asymptotically similar.
\end{prop}

\begin{proof}
To prove $B\in \overline{\mathcal{S}(A)}$ it is enough to construct invertible operators $W_n$ such that 
$\lim_{n}{\norm{B-W_nAW_{n}^{-1}}}=0$. For this, consider the sequence $\lambda_{n}:=\frac{a_n}{b_n}$ and the $2\times2$ matrices $D_{\lambda_n}:=P+\lambda_{n}Q$. So if we define the invertible element $W_n$ by
$$W_{n}:=D_{\lambda_1}\otimes D_{\lambda_2}\otimes\ldots\otimes D_{\lambda_n}\in \mathcal{R}$$
it is easy to see that if $A_n=\sum_{k=1}^{n}{a_kV_k}$ and $B_n=\sum_{k=1}^{n}{b_kV_k}$ then $W_nA_nW_n^{-1}=B_n$ and 
$W_nAW_n^{-1}=B_n+A-A_{n}$. Since $\lim_{n}{\norm{B-B_n}}=0$ and $\lim_{n}{\norm{A-A_n}}=0$ we see that $\lim_{n}{\norm{B-W_nAW_{n}^{-1}}}=0$. A similar argument shows that $A\in \overline{\mathcal{S}(B)}$ and concludes the proof.
\end{proof}

\vspace{0.4cm}

\begin{obs}
Let $\{a_n\}_n$ in $l_1(\N)$ and $A$ as before. We will show that $A$ is a commutant operator, i.e.: there exist $B$ and $W$ such that
$A=[W,B]$. It is clear that we can choose $\{b_n\}_n\in l_1(\N)$ such that $b_n>0$ and $\sum_{n=1}^{+\infty}{\frac{|a_n|}{b_n}}<+\infty$. Let $B:=\sum_{n=1}^{+\infty}{b_n\,V_n}$ and $W:=\sum_{n=1}^{+\infty}{\frac{a_n}{b_n}\,P_n}$. 
Since $P_{n}V_{n}=V_{n}$ and $V_{n}P_{n}=0$ it is easy to see that
$$WB-BW=[W,B]=A.$$
\end{obs}

\vspace{0.6cm}
\section{Haagerup's invariant subspaces}

\par As we described in the introduction, given an operator $T$ in a $\mathrm{II}_1$ factor $\mathcal{M}$, Haagerup and Schultz \cite{Ha} constructed for each Borel set $B$ in the complex plane an invariant subspace affiliated to the von Neumann algebra generated by $T$, such that 
$\tau(P_B)=\mu(B)$. If the Borel set $B$ is a closed ball of radius $r$ centered at $\lambda$. Then $\mathcal{K}_T(B)$ is the set of vectors $\xi\in\mathcal{H}$, for which there is a sequence $\{\xi_n\}_{n}$ in $\mathcal{H}$ such that
$$\lim_{n}{\norm{\xi_n-\xi}}=0 \textrm{\quad and \quad} \limsup_{n}{\norm{(T-\lambda 1)^{n}\xi_n}^{\frac{1}{n}}}\leq r.$$

\noindent For any sequence $\{\gamma_n\}_{n=1}^{+\infty}$ of positive numbers and $r>0$, we define a subspace $\mathcal{H}_{r}(T)$ (similar to the one considered in \cite{DT1} to prove that the quasinilpotent DT-operator has non-trivial hyperinvariant subspaces) by 

%\begin{equation}\label{Hr}
%\lim_{n}{\norm{\xi_n-\xi}}=0 \textrm{\quad and \quad} \limsup_{n}{\gamma_n\norm{T^{n}\xi_n}^{\frac{1}{n}}}\leq r.
%\end{equation}

\begin{equation}\label{Hr}
\mathcal{E}_{r}(T):=\{\xi\in\mathcal{H}\,:\, \limsup_{n}{\,\gamma_n\norm{T^{n}(\xi)}^{1/n}}\leq r\}\quad \textrm{and}\quad  \mathcal{H}_{r}(T)=\overline{\mathcal{E}_{r}(T)}.
\end{equation}

\noindent This subspace is closed, $T$-invariant, affiliated to the von Neumann algebra, and moreover, hyperinvariant. However, we will prove that for any sequence $\{\gamma_{n}\}_{n}$ this subspace is trivial. Let $0<\alpha<1$ and consider the operator 
\begin{equation}\label{Aa}
A:=\sum_{n=1}^{+\infty}{\alpha^{n} \, V_n} \mathrm{\,\,\,\, where\,\,\,\,} V_n=I^{\otimes (n-1)}\otimes V.
\end{equation}

\vspace{0.3cm}

\begin{notation}\label{nott}
Given a word $w$ with letters in the alphabet $\{P,Q,V,V^*\}$ we can associate an element in $\mathcal{H}$ by adding a tensor product between each of the letters. For example, the word $VPQV^{*}PQV$ corresponds to the tensor word $V\otimes P\otimes Q\otimes V^{*}\otimes P\otimes Q\otimes V$ which is a vector in our Hilbert space. Note that if $w$ is a tensor word as before then $ww^{*}$ is a word with letters in $P$ and $Q$ only. We define the symbol $\#_{P}(ww^{*})$ as the number of $P$'s in the word $ww^{*}$.\\
\end{notation}
\vspace{0.3cm}

\begin{prop}\label{lemma} 
Let $w=P^{\otimes k_1}\otimes Q^{\otimes r_1}\otimes \ldots \otimes P^{\otimes k_n}\otimes Q^{\otimes r_n}$
where $n\geq 1$ and $k_i,\,r_i\geq 0$ for $i=1,\ldots,n$. Then for $A$ as in (\ref{Aa}) we have
\begin{equation}\label{w}
\lim_{m}{\Bigg( \frac{\norm{A^{m}\hat{w}}_2}{\norm{A^{m}\hat{1}}_2} \Bigg)^{\frac{1}{m}}}=\alpha^{k_1+k_2+\ldots+k_n}.
\end{equation}
\end{prop}

\begin{proof}
Let $m\geq 1$ then
\begin{equation}
(A^*)^mA^m=\sum_{1\leq p_1,\,q_1,\ldots,\,p_m,\,q_m}{\alpha^{p_1}\alpha^{q_1}\ldots\alpha^{p_m}\alpha^{q_m}V_{q_1}^{*}\ldots 
V_{q_m}^{*}V_{p_1}\ldots V_{p_m}}
\end{equation}
and $\norm{A^{m}\hat{w}}_2^{2}=\tau(w(A^{*})^mA^mw)=\tau((A^{*})^mA^mw)$.\\

\noindent Note that $(A^*)^mA^m=R_{m}+T_{m}$ where 
\begin{equation}\label{Rm}
R_m:=(m!)^2\sum_{1\leq q_1<q_2<\ldots<q_m}{\alpha^{2q_1}\alpha^{2q_2}\ldots \alpha^{2q_m}Q_{q_1}\ldots Q_{q_m}} \text{\quad and \quad} T_{m}=(A^*)^mA^m-R_m.
\end{equation}
Since $VP=0$, $VQ=V$, $V^{*}P=V^{*}$, $V^{*}Q=0$ and $\tau(V)=\tau(V^{*})=0$ it is easy to see that for any word $w$, with letters in $\{P,\,Q\}$, then $\tau(T_{m}w)=0$. Hence, $$\norm{A^{m}\hat{w}}_2^{2}=\tau((A^{*})^mA^mw)=\tau(R_mw).$$

\noindent Now we will proceed by induction, the case $w=1$ is obvious. Assume that the statement is true for any word $w$ 
of length $r\geq 1$. We will prove it for $P\otimes w$  and for $Q\otimes w$ and we will be done. Consider first the case $P\otimes w$. 
Then
$$\norm{A^{m}(\hat{P\otimes w})}_2^{2}=\tau(R_m(P\otimes w))$$
\begin{eqnarray*}
\tau(R_m(P\otimes w)) & = & (m!)^2\tau\Bigg( \sum_{1\leq q_1<q_2<\ldots<q_m}{\alpha^{2q_1}\ldots \alpha^{2q_m}Q_{q_1}\ldots Q_{q_m}(P\otimes w)}\Bigg)\\
& = & (m!)^2\tau\Bigg( \sum_{2\leq q_1<q_2<\ldots<q_m}{\alpha^{2q_1}\ldots \alpha^{2q_m}Q_{q_1}\ldots Q_{q_m}(P\otimes w)}\Bigg)\\
& = & \frac{1}{2}\alpha^{2m}(m!)^2\tau\Bigg( \sum_{1\leq q_1<q_2<\ldots<q_m}{\alpha^{2q_1}\ldots \alpha^{2q_m}Q_{q_1}\ldots Q_{q_m}w}\Bigg)\\
& = & \frac{1}{2}\alpha^{2m}\tau(R_mw).
\end{eqnarray*}
Therefore, $\norm{A^{m}(\hat{P\otimes w})}_2^{2}=\frac{1}{2}\alpha^{2m}\norm{A^{m}\hat{w}}_2^{2}\,\,\,$ and hence
$$\lim_{m}{\Bigg( \frac{\norm{A^{m}(\hat{P\otimes w})}_2}{\norm{A^{m}\hat{1}}_2} \Bigg)^{1/m}}=\alpha\cdot \lim_{m}{\Bigg( \frac{\norm{A^{m}\hat{w}}_2}{\norm{A^{m}\hat{1}}_2} \Bigg)^{1/m}}.$$
We are done with the case $P\otimes w$ by the induction hypothesis and the fact that the number of $P$'s in $P\otimes w$ is the same as the number in $w$ plus one. Let us consider the case $Q\otimes w$. First note that 
$$\tau(R_m(1\otimes w))=\tau(R_m(P\otimes w))+\tau(R_m(Q\otimes w))=\frac{1}{2}\alpha^{2m}\tau(R_{m}w)+\tau(R_m(Q\otimes w))$$
and
\begin{eqnarray*}
\tau(R_m(1\otimes w)) & = & (m!)^2\tau\Bigg(\sum_{1=q_1<q_2<\ldots<q_m}{\alpha^{2}\alpha^{2q_2}\ldots\alpha^{2q_m}Q_{1}Q_{q_2}\ldots Q_{q_m}(1\otimes w)}\Bigg) +\\
& + &(m!)^2\tau\Bigg(\sum_{2\leq q_1<q_2<\ldots<q_m}{\alpha^{2q_1}\alpha^{2q_2}\ldots\alpha^{2q_m}Q_{q_1}Q_{q_2}\ldots Q_{q_m}(1\otimes w)}\Bigg).
\end{eqnarray*}
Therefore,
$$\tau(R_m(1\otimes w))=\frac{\alpha^{2m}m^2}{2}\tau(R_{m-1}w)+\alpha^{2m}\tau(R_{m}w).$$
Thus,
$$\tau(R_m(Q\otimes w))=\frac{\alpha^{2m}}{2}\Big (m^2\tau(R_{m-1}w)+\tau(R_{m}w)\Big )\geq \frac{\alpha^{2m}}{2}m^2\tau(R_{m-1}w).$$
Hence,
\begin{eqnarray*}
\liminf_{m}{\Bigg( \frac{\norm{A^{m}(\hat{Q\otimes w})}_2}{\norm{A^{m}\hat{1}}_2} \Bigg)^{\frac{1}{m}}} & = & \liminf_{m}{\Bigg( \frac{\tau(R_m(Q\otimes w))}{\tau(R_m)}\Bigg )^{\frac{1}{2m}}}\\
& \geq & \liminf_{m}{\Bigg( \frac{\alpha^{2m}m^2\tau(R_{m-1}w)}{2\,\,\tau(R_m)}\Bigg )^{\frac{1}{2m}}}.
\end{eqnarray*}
Since $\tau(R_m)=\frac{m^{2}\alpha^{2m}}{2(1-\alpha^{2m})}\tau(R_{m-1})$ we obtain that
\begin{eqnarray*}
\liminf_{m}{\Bigg( \frac{\norm{A^{m}(\hat{Q\otimes w})}_2}{\norm{A^{m}\hat{1}}_2} \Bigg)^{\frac{1}{m}}} & \geq & \liminf_{m}{(1-\alpha^{2m})^{\frac{1}{2m}}\Bigg( \frac{\tau(R_{m-1}w)}{\tau(R_{m-1})}\Bigg )^{\frac{1}{2m}}}\\
& = & \lim_{m}{\Bigg( \frac{\norm{A^{m}\hat{w}}_2}{\norm{A^{m}\hat{1}}_2} \Bigg)^{\frac{1}{m}}}.
\end{eqnarray*}
Observing now that since $\tau(R_m(Q\otimes w))=\frac{\alpha^{2m}}{2}\Big (m^2\tau(R_{m-1}w)+\tau(R_{m}w)\Big )$ then 
\begin{eqnarray*}
\norm{A^{m}(\hat{Q\otimes w})}_{2}^{2} & = &\frac{\alpha^{2m}}{2}\Bigg( m^{2}\norm{A^{m-1}\hat{w}}_{2}^{2}+\norm{A^{m}\hat{w}}_{2}^{2}\Bigg)\\
& \leq & \frac{\alpha^{2m}}{2}\Bigg( m^{2}\norm{A^{m-1}\hat{w}}_{2}^{2}+\norm{A}^{2}_{\infty}\norm{A^{m-1}\hat{w}}_{2}^{2}\Bigg)\\
& = & \frac{\alpha^{2m}}{2}\norm{A^{m-1}\hat{w}}_{2}^{2}\Big(m^{2}+\norm{A}_{\infty}^{2}\Big).
\end{eqnarray*}
Therefore,
\begin{eqnarray*}
\limsup_{m}{\Bigg( \frac{\norm{A^{m}(\hat{Q\otimes w})}_2}{\norm{A^{m}\hat{1}}_2} \Bigg)^{\frac{1}{m}}} & \leq & \limsup_{m}{\Bigg((1-\alpha^{2m})\,\frac{m^{2}+\norm{A}^{2}_{\infty}}{m^2}\Bigg)^{\frac{1}{2m}} \Bigg( \frac{\norm{A^{m}\hat{w}}_2}{\norm{A^{m}\hat{1}}_2} \Bigg)^{\frac{1}{m}}}\\
& = & \lim_{m}{\Bigg( \frac{\norm{A^{m}\hat{w}}_2}{\norm{A^{m}\hat{1}}_2} \Bigg)^{\frac{1}{m}}}.
\end{eqnarray*}
Hence,
$$\lim_{m}{\Bigg( \frac{\norm{A^{m}(\hat{Q\otimes w})}_2}{\norm{A^{m}\hat{1}}_2} \Bigg)^{\frac{1}{m}}}=\lim_{m}{\Bigg( \frac{\norm{A^{m}\hat{w}}_2}{\norm{A^{m}\hat{1}}_2} \Bigg)^{\frac{1}{m}}}$$
which concludes the proof.
\end{proof}

\vspace{0.4cm}

\noindent Since $\norm{A^{m}\hat{w}}_2^{2}=\tau(w^{*}(A^{*})^mA^mw)=\tau((A^{*})^mA^mww^{*})$ then with the notation in \ref{nott}, Proposition \ref{lemma} says that
\begin{equation}\label{wo}
\lim_{m}{\Bigg( \frac{\norm{A^{m}\hat{w}}_2}{\norm{A^{m}\hat{1}}_2} \Bigg)^{\frac{1}{m}}}=\alpha^{\#_{P}(ww^{*})}.
\end{equation}

\vspace{0.3cm}

\begin{prop} Let $n\geq 1$ and $\xi=\sum_{i=1}^{n}{c_{i}w_{i}}$ be a vector with $c_{i}\in \C$ and $w_{i}$ tensor words of length $r_{i}$ for $i=1,\ldots, n$. Then
\begin{equation}
\liminf_{m}{\Bigg( \frac{\norm{A^{m}\xi}_2}{\norm{A^{m}1}_2} \Bigg)^{\frac{1}{m}}}\geq \frac{1}{\sqrt{2}}\alpha^{r} \quad\quad \textrm{where}\quad r=\max\{r_{i}\,:\,i=1,\ldots, n\}.
\end{equation}
\end{prop}

\vspace{0.3cm}

\begin{proof}
Let $\xi=\sum_{i=1}^{n}{c_{i}w_{i}}$ and $r=\max\{r_{i}\,:\,i=1,\ldots, n\}$ then 

\begin{eqnarray*}
A^{m}\xi & = & m!\sum_{1\leq p_1<\ldots <p_{m}}{\alpha^{p_1}\ldots \alpha^{p_m}V_{p_1}\ldots V_{p_m}\xi}\\
& = & m!\sum_{J_{m}}{\alpha^{p_1}\ldots \alpha^{p_m}V_{p_1}\ldots V_{p_m}\xi}+m!\sum_{r+1\leq p_1<\ldots<p_{m}}{\alpha^{p_1}\ldots \alpha^{p_m}V_{p_1}\ldots V_{p_m}\xi}\\
& = & m!\sum_{J_{m}}{\alpha^{p_1}\ldots \alpha^{p_m}V_{p_1}\ldots V_{p_m}\xi}+m!\sum_{r+1\leq p_1<\ldots<p_{m}}{\alpha^{p_1}\ldots \alpha^{p_m}\xi V_{p_1}\ldots V_{p_m}}
\end{eqnarray*}
where $J_{m}:=\{1\leq p_1<p_2<\ldots <p_{m} \,:\, \textrm{such that exists $i$ so that}\,\, p_{i}\leq r\}$.\\ 

\vspace{0.1cm}
\noindent It is easy to see that for $r+1\leq p_1<\ldots<p_{m}$ the vectors $\xi V_{p_1}\ldots V_{p_m}$ are pairwise orthogonal and are orthogonal to $A_{r}^{(m)}\xi:=m!\sum_{J_{m}}{\alpha^{p_1}\ldots \alpha^{p_m}V_{p_1}\ldots V_{p_m}\xi}$. Note also that $\norm{\xi V_{p_1}\ldots V_{p_m}}_2^{2}=\frac{\norm{\xi}_2^{2}}{2^m}$ for $r+1\leq p_1<\ldots<p_{m}$.\\

\noindent Therefore,
\begin{eqnarray*}
\norm{A^{m}\xi}_2^{2} & = &\norm{A_{r}^{(m)}\xi}_2^{2}+\frac{\norm{\xi}_2^{2}(m!)^2}{2^m}\sum_{r+1\leq p_1<\ldots<p_{m}}{\alpha^{2p_1}\ldots\alpha^{2p_m}}\\
& \geq & \frac{\norm{\xi}_2^{2}(m!)^2\alpha^{2rm}}{2^m}\sum_{1\leq p_1<\ldots<p_{m}}{\alpha^{2p_1}\ldots\alpha^{2p_m}}\\
& = & \frac{\norm{\xi}_{2}^{2}}{2^m}\alpha^{2rm}\norm{A^{m}1}_{2}^{2}.
\end{eqnarray*}
Hence,
$$\liminf_{m}{\Bigg( \frac{\norm{A^{m}\xi}}{\norm{A^{m}1}} \Bigg)^{\frac{1}{m}}}\geq \frac{1}{\sqrt{2}}\alpha^{r}\quad\quad \textrm{where}\quad r=\max\{r_{i}\,:\,i=1,\ldots, n\}.$$
\end{proof}

\vspace{0.2cm}

\begin{teo}\label{Ker}
The operator $A$ has trivial kernel and dense range.
\end{teo}

\vspace{0.2cm}

\begin{proof}
Since this operator lives in a finite factor it is enough to prove that $\mathrm{Ker}(A)=\{0\}$. Let us consider the Hilbert space $\mathcal{H}\oplus \mathcal{H}$ and the operator $\tilde{A}:\mathcal{H}\oplus \mathcal{H}\to \mathcal{H}\oplus \mathcal{H}$ given by

\vspace{0.2cm}
\begin{equation*}
\tilde{A}=\left( \begin{array}{cc}
\alpha A & \alpha\\
0 & \alpha A\\
\end{array} \right). 
\end{equation*}
\vspace{0.2cm}

\noindent Decompose $\mathcal{H}$ as $\mathcal{H}=\mathcal{H}_1\oplus \mathcal{H}_2$ where 
$\mathcal{H}_1:=\{\hat{Q}\otimes\xi + \hat{V}^{*}\otimes\eta \,\,\,:\,\,\, \xi,\eta\in \mathcal{H}\}$ and $\mathcal{H}_2:=\{\hat{P}\otimes\xi + \hat{V}\otimes\eta \,\,\,:\,\,\, \xi,\eta\in \mathcal{H}\}$ and $Q$ and $P$ are the orthogonal projections onto these subspaces respectively. Since 
$$PAP=\sum_{n=2}^{+\infty}{\alpha^{n}P\otimes I^{\otimes(n-1)}\otimes V},\quad PAQ=\alpha V$$
$$QAQ=\sum_{n=2}^{+\infty}{\alpha^{n}Q\otimes I^{\otimes(n-1)}\otimes V},\quad QAP=0$$ 
then $A$ has trivial kernel if and only if $\tilde{A}$ has trivial kernel and moreover, $\tau_{\mathcal{R}}(\mathrm{Ker}(A^{n}))=\tau_{\mathcal{R}\otimes M_{2}(\C)}(\mathrm{Ker}(\tilde{A}^{n})):=\gamma_{n}$.\\

\noindent It is easy to see that 
\begin{equation*}
\tilde{A}^{n}=\left( \begin{array}{cc}
\alpha^{n} A^{n} & n\alpha^{n} A^{n-1}\\
0 & \alpha^{n} A^{n}\\
\end{array} \right)  
\end{equation*}
and since 
\begin{equation*}
\left( \begin{array}{cc}
\alpha^{n} A^{n} & n\alpha^{n} A^{n-1}\\
0 & \alpha^{n} A^{n}\\
\end{array} \right)  
\left( \begin{array}{c}
\xi_{1}\\
\xi_{2}\\
\end{array}\right )=
\left( \begin{array}{c}
\alpha^{n}(A^{n}(\xi_{1})+nA^{n-1}(\xi_2))\\
\alpha^{n}A^{n}(\xi_{2})\\
\end{array}\right )
\end{equation*}

\vspace{0.2cm}
\noindent we see that $\mathrm{Ker}(\tilde{A}^n)=\Big\{(\xi,-\frac{1}{n}A(\xi)+\eta)\,\,\,:\,\,\, \xi\in\mathrm{Ker}(A^{n+1}),\,\, \eta\in\mathrm{Ker}(A^{n-1})\Big\}$.\\

\noindent Hence 

$$\tau_{\mathcal{R}\otimes M_{2}(\C)}(\mathrm{Ker}(\tilde{A}^{n}))=\frac{1}{2}\Big(\tau_{\mathcal{R}}(\mathrm{Ker}(A^{n+1}))+\tau_{\mathcal{R}}(\mathrm{Ker}(A^{n-1}))\Big).$$

\vspace{0.2cm}

\noindent Therefore, $\gamma_{n}=\frac{1}{2}(\gamma_{n+1}+\gamma_{n-1})$ which implies that $\gamma_{n}=n\gamma_{1}$. Therefore, $\gamma_{1}=0$ and thus $\mathrm{Ker}(A)=\{0\}$.

\end{proof}

%\begin{notation}\label{not}
%Given $\xi\in \mathcal{H}$ we can represent this vector as $\xi=\sum_{n}{\lambda_{n}w_{n}}$ where $\lambda_{n}\in\C$ and $w_{n}$ is a tensor word %with letters in $\{P,Q,V,V^{*}\}$ for each $n$. Now given $w$ and $v$ two tensor words we will denote by 
%$$v\otimes\xi:=\sum_{n}{\lambda_{n}v\otimes w_n}\quad \text{and}\quad v\xi w:=\sum_{n}{\lambda_{n}vw_{n}w}.$$
%\end{notation}

\vspace{0.3cm}

\begin{teo} Let $r>0$ and $\{\gamma_n\}_{n=1}^{+\infty}$ be a sequence of positive numbers and $A$ be as in (\ref{Aa}). 
The subspace $\mathcal{H}_{r}(A)$ defined by $\mathcal{H}_{r}(A):=\overline{\mathcal{E}_{r}(A)}$ where

$$\mathcal{E}_{r}(A):=\{\xi\in\mathcal{H}\,:\,\limsup_{n}{\gamma_n\,\norm{A^{n}\xi}_2^{\frac{1}{n}}\leq r}\}$$

\noindent is either $\mathcal{H}$ or $\{0\}$.

\end{teo}

\vspace{0.2cm}

\begin{proof}
Decompose $\mathcal{H}$ as $\mathcal{H}=\mathcal{H}_1\oplus \mathcal{H}_2$ where 
$\mathcal{H}_1:=\{\hat{Q}\otimes\xi + \hat{V}^{*}\otimes\eta \,\,\,:\,\,\, \xi,\eta\in \mathcal{H}\}$ and $\mathcal{H}_2:=\{\hat{P}\otimes\xi + \hat{V}\otimes\eta \,\,\,:\,\,\, \xi,\eta\in \mathcal{H}\}$ as we did in Theorem \ref{Ker}. Then the operator $A$ can be represented as $A:\mathcal{H}_1\oplus \mathcal{H}_2\to \mathcal{H}_1\oplus \mathcal{H}_2$ with
\vspace{0.2cm}
\begin{equation}\label{repp}
A=\left( \begin{array}{cc}
\alpha A & \alpha\\
0 & \alpha A\\
\end{array} \right) 
\end{equation}
and hence
\begin{equation*}
A^{n}=\left( \begin{array}{cc}
\alpha^{n}A^{n} & n\alpha^{n}A^{n-1}\\
0 & \alpha^{n} A^{n}\\
\end{array} \right). 
\end{equation*}

\vspace{0.2cm}
\noindent Therefore, under the canonical isomorphism of $\mathcal{R}\simeq M_{2}(\C)\otimes \mathcal{R}$ we see that the operator $A$ is identified with $\alpha(P+Q)\otimes A+\alpha V\otimes 1$ and $A^{n}$ is identified with $\alpha^{n}(P+Q)\otimes A^{n}+n\alpha^{n}V\otimes A^{n-1}$.

\vspace{0.3cm}
\noindent The subspace $\mathcal{H}_{r}(A)$ is hyperinvariant, hence, affiliated to the von Neumann algebra $\mathcal{R}$. Let $\beta$ be the trace of this subspace $\beta=\tau(\mathcal{H}_{r}(A))$.
Define the subspaces $E_{1}:=\{\hat{P}\otimes A(\xi)+\hat{V}\otimes A(\eta) \,\,\,:\,\,\, \xi,\eta\in \mathcal{E}_{r}(A)\}$, $E_{2}:=\{\hat{Q}\otimes A(\xi)+\hat{V}^{*}\otimes A(\eta) \,\,\,:\,\,\, \xi,\eta\in \mathcal{E}_{r}(A)\}$, $H_{1}=\overline{E_{1}}$ and $H_{2}=\overline{E_{2}}$. The subspaces $H_{1}$ and $H_{2}$ are affiliated to $\mathcal{R}$ and since the kernel of $A$ is trivial $\tau(H_1)=\tau(H_2)=\frac{\beta}{2}$. It is clear that the subspaces $H_{1}$ and $H_{2}$ are orthogonal. Now we will prove that $E_{1}, E_{2}\subset \mathcal{E}_{r}(A)$ and hence $\mathcal{H}_{r}(A)=H_{1}\oplus H_{2}$. Let $\xi$ and $\eta$ be vectors in $\mathcal{E}_{r}(A)$ and $h=\hat{P}\otimes A(\xi)+\hat{V}\otimes A(\eta)\in E_{1}$ then

\begin{eqnarray*}
\norm{A^{n}(h)}_{2} & = & \norm{\alpha^{n}(P+Q)\otimes A^{n}+n\alpha^{n}V\otimes A^{n-1}(\hat{P}\otimes A(\xi)+\hat{V}\otimes A(\eta))}_{2}\\
& = & \norm{\alpha^{n}(\hat{P}\otimes A^{n+1}(\xi)+\hat{V}\otimes A^{n+1}(\eta))}_{2}\\
& \leq & 2 \alpha^{n} \cdot \sup\{ \norm{\hat{P}\otimes A^{n+1}(\xi)}_{2},\norm{\hat{V}\otimes A^{n+1}(\eta)}_{2}\}\\
& \leq & \sqrt{2} \alpha^{n} \norm{A} \cdot \sup\{ \norm{A^{n}(\xi)}_2,\norm{A^{n}(\eta)}_2 \}.
\end{eqnarray*}

\noindent Therefore,

\begin{equation*}
\limsup_{n}{\,\gamma_{n} \norm{A^{n}(\hat{P}\otimes A(\xi)+\hat{V}\otimes A(\eta))}_{2}^{\frac{1}{n}}}\leq \alpha\,r < r.
\end{equation*}

\vspace{0.5cm}

\noindent Thus, $E_{1}\subset \mathcal{E}_{r}(A)$. Analogously, let $\xi$ and $\eta$ be vectors in $\mathcal{E}_{r}(A)$ and $h=\hat{Q}\otimes A(\xi)+\hat{V}^{*}\otimes A(\eta)\in E_{2}$ then

\begin{eqnarray*}
\norm{A^{n}(h)}_{2} & = & \norm{\alpha^{n}(P+Q)\otimes A^{n}+n\alpha^{n}V\otimes A^{n-1}(\hat{Q}\otimes A(\xi)+\hat{V}^{*}\otimes A(\eta))}_{2}\\
& = & \alpha^{n} \norm{\hat{Q}\otimes A^{n+1}(\xi)+\hat{V}^{*}\otimes A^{n+1}(\eta)+n\hat{V}\otimes A^{n}(\xi)+n\hat{P}\otimes A^{n}(\eta)}_{2}\\
& \leq & \sqrt{2}\alpha^{n}(n+\norm{A}) \cdot \sup\{\norm{A^{n}(\xi)}_{2},\norm{A^{n}(\eta)}_{2}\}.
\end{eqnarray*}

\noindent Therefore,

\begin{equation*}
\limsup_{n}{\,\gamma_{n} \norm{A^{n}(\hat{Q}\otimes A(\xi)+\hat{V}^{*}\otimes A(\eta))}_{2}^{\frac{1}{n}}}\leq \alpha\,r<r.
\end{equation*}

\vspace{0.3cm}
\noindent Hence, $E_{2}\subset \mathcal{E}_{r}(A)$ and therefore, $\mathcal{H}_{r}(A)=H_{1}\oplus H_{2}$. Since $V^{*}(E_{1})\subseteq E_{2}$ and $V^{*}(E_{2})=\{0\}$ we see that $\mathcal{H}_{r}(A)$ is $V^{*}$--invariant.\\ 

\vspace{0.2cm}

\noindent Representing now our operator $A$ as 

\begin{equation*}
A=\left( \begin{array}{cccc}
\alpha^2 A & \alpha^2 & \alpha & 0 \\
0 & \alpha^2 A & 0 & \alpha \\
0 & 0 & \alpha^2 A & \alpha^2 \\
0 & 0 & 0 & \alpha^{2}A^{2} \\
\end{array} \right)
\end{equation*}

\vspace{0.2cm}
\noindent it is not hard to see that

\begin{equation*}
A^{n}=\left( \begin{array}{cccc}
\alpha^{2n} A^{n} & \alpha^{2n}nA^{n-1} & \alpha^{2n-1}nA^{n-1} & \alpha^{2n-1}n(n-1) A^{n-2} \\
0 & \alpha^{2n} A^{n} & 0 & \alpha^{2n-1}n A^{n-1} \\
0 & 0 & \alpha^{2n} A^{n} & \alpha^{2n}n A^{n-1} \\
0 & 0 & 0 & \alpha^{2n}A^{n} \\
\end{array} \right).
\end{equation*}

\vspace{0.5cm}

\noindent Define the subspaces $E_{11}:=\{\hat{P}\otimes\hat{P}\otimes A^{2}(\xi_1)+\hat{P}\otimes\hat{V}\otimes A^{2}(\xi_2)+\hat{V}\otimes\hat{P}\otimes A^{2}(\xi_3)+\hat{V}\otimes\hat{V}\otimes A^{2}(\xi_4)\,\,:\,\, \xi_{1},\xi_{2},\xi_{3},\xi_{4}\in \mathcal{E}_{r}(A)\}$, $E_{12}:=\{\hat{P}\otimes\hat{Q}\otimes A^{2}(\xi_1)+\hat{P}\otimes\hat{V}^{*}\otimes A^{2}(\xi_2)+\hat{V}\otimes\hat{Q}\otimes A^{2}(\xi_3)+\hat{V}\otimes\hat{V}^{*}\otimes A^{2}(\xi_4)\,\,:\,\, \xi_{1},\xi_{2},\xi_{3},\xi_{4}\in \mathcal{E}_{r}(A)\}$, $E_{21}:=\{\hat{Q}\otimes\hat{P}\otimes A^{2}(\xi_1)+\hat{Q}\otimes\hat{V}\otimes A^{2}(\xi_2)+\hat{V}^{*}\otimes\hat{P}\otimes A^{2}(\xi_3)+\hat{V}^{*}\otimes\hat{V}\otimes A^{2}(\xi_4)\,\,:\,\, \xi_{1},\xi_{2},\xi_{3},\xi_{4}\in \mathcal{E}_{r}(A)\}$, $E_{22}:=\{\hat{Q}\otimes\hat{Q}\otimes A^{2}(\xi_1)+\hat{Q}\otimes\hat{V}^{*}\otimes A^{2}(\xi_2)+\hat{V}^{*}\otimes\hat{Q}\otimes A^{2}(\xi_3)+\hat{V}^{*}\otimes\hat{V}^{*}\otimes A^{2}(\xi_4)\,\,:\,\, \xi_{1},\xi_{2},\xi_{3},\xi_{4}\in \mathcal{E}_{r}(A)\}$
and $H_{ij}=\overline{E_{i,j}}$ for $i,j=1,2$. Using the same argument as before it is not hard to see that $\mathcal{H}_{r}(A)=H_{11}\oplus H_{12}\oplus H_{21}\oplus H_{22}$ and that $\mathcal{H}_{r}(A)$ is $I\otimes V^{*}$--invariant. 

\vspace{0.5cm}
\noindent Analogously, given $n\geq 1$ and $i_{1},\ldots,i_{n}\in \{1,2\}$ we define $\Omega_{i_1,\ldots,i_{n}}$ the set of words of length $n$ in the alphabet $\{\hat{P},\hat{Q},\hat{V}^{*},\hat{V}\}$ given by 

$$\Omega_{i_1,\ldots,i_{n}}=\{a_{1}\otimes\ldots \otimes a_{n} \,\,:\,\,a_{k}\in \{\hat{P},\hat{V}\} \textrm{\,\,if\,\,}\, i_{k}=1 \textrm{\,\,\,and\,\,\,} a_{k}\in \{\hat{Q},\hat{V}^{*}\} \textrm{\,\,\,if\,\,\,} i_{k}=2 \}.$$

\vspace{0.3cm}
\noindent Let $E_{i_1,\ldots,i_{n}}$ and $H_{i_1,\ldots,i_{n}}$ be the subspaces defined by

$$E_{i_1,\ldots,i_{n}}:=\mathrm{span}\{w\otimes A^{n}(\xi)\,\,:\,\, w\in\Omega_{i_1,\ldots,i_{n}},\, \xi\in\mathcal{E}_{r}(A)\}\quad\textrm{and}\quad H_{i_1,\ldots,i_{n}}=\overline{E_{i_1,\ldots,i_n}}.$$

\vspace{0.3cm}
\noindent We can see that 

$$\mathcal{H}_{r}(A)=\bigoplus_{\{i_1,\ldots,i_{n}\}\in \{1,2\}^{n}}{H_{i_1,\ldots,i_{n}}}$$

\noindent and therefore, $\mathcal{H}_{r}(A)$ is $I^{\otimes (n-1)}\otimes V^{*}$--invariant. Since $n$ is arbitrary we see that $\mathcal{H}_{r}(A)$ is $A^{*}$--invariant. Thus, the subspace $\mathcal{H}_{r}(A)$ is $A$--invariant and $A^{*}$--invariant and hence trivial.

\end{proof}

%\begin{obs}
%By Lemma \ref{E0} we know that the subspace $H_{0}$ is either $\{0\}$ or $\mathcal{H}$. If we could prove that  
%for each nonzero vector $\xi\in\mathcal{H}$ there exists a constant $K(\xi)>0$ such that %$$\liminf_{m}{\Bigg(\frac{\norm{A^{m}\xi}_{2}}{\norm{A^{m}\hat{1}}_{2}}\Bigg)^{\frac{1}{m}}}=K(\xi)>0.$$

%\vspace{0.2cm}
%\noindent Then, if $\,\,r>0$ and $\{\gamma_{m}\}_{m}$ is a sequence of positive real numbers such that %$$\limsup_{m}{\gamma_{m}\norm{A^m\hat{1}}_{2}^{\frac{1}{m}}}=+\infty$$ 
%then
%\begin{eqnarray*}
%\limsup_{m}{\gamma_m\norm{A^{m}\xi}_2^{\frac{1}{m}}} & = & \limsup_{m}{\gamma_m\norm{A^{m}\hat{1}}_2^{\frac{1}{m}}\Bigg( %\frac{\norm{A^{m}\xi}_2}{\norm{A^{m}1}_2}\Bigg )^{\frac{1}{m}}}\\
%& \geq & K(\xi)\cdot\limsup_{m}{\gamma_{m}\norm{A^m\hat{1}}_{2}^{\frac{1}{m}}}=+\infty.
%\end{eqnarray*}
%Hence, $\mathcal{E}_{r}(A)=\{0\}$ and then $\mathcal{H}_{r}(A)=\{0\}$.
%\end{obs}

\vspace{0.6cm}

\begin{question}
Does the operator $A$ have non-trivial hyperinvariant subspaces?
\end{question}

\vspace{0.6cm}

\section{Distribution of $\mathrm{Re}(A)$ and $\mathrm{Im}(A)$}

In this section we will prove that given $\{c_n\}_n\in l_{1}(\N)$ with $c_n\geq 0$, then $\mathrm{Re}(A)$ and $\mathrm{Im}(A)$ have the same distribution and we will describe its moments. Let $X=A+A^*$ and $Y=A-A^*$, then $\mathrm{Re}(A)=\frac{1}{2}X$ and $\mathrm{Im}(A)=\frac{1}{2i}Y$. Thus,

\begin{equation}
X=\sum_{n=1}^{+\infty}{c_n\,R_n}  \mathrm{\quad where \quad} R_n=I^{\otimes(n-1)}\otimes R \mathrm{\quad with \quad} 
R=\left( \begin{array}{cc}
0 & 1\\
1 & 0\\
\end{array} \right)
\end{equation}

\begin{equation}
Y=\sum_{n=1}^{+\infty}{c_n\,T_n}  \mathrm{\quad\,\,\, where \quad} T_n=I^{\otimes(n-1)}\otimes T \mathrm{\quad with \quad} 
T=\left( \begin{array}{cc}
0 & 1\\
-1 & 0\\
\end{array} \right)
\end{equation}

\vspace{0.3cm}

\noindent Note that $R^2=1$, $T^2=-1$ and $\tau(R)=\tau(T)=0$. From this observation, it is clear that $\tau(X^{2p+1})=\tau(Y^{2p+1})=0$. Now we will find a combinatorial formula for $\tau(X^{2p})$ and prove that $\tau(Y^{2p})=(-1)^p\tau(X^{2p})$ for $p\geq 0$. But first we will fix some notation. Given $p\geq 0$, $1\leq k \leq p$ and $n_1\geq n_2\geq \ldots\geq n_k$ such that $\sum_{i=1}^{k}{n_i}=p$ we will denote by $\gamma(p\,;n_1,n_2,\ldots, n_k)$ the number of partitions of the set $\{1,2,\ldots 2p\}$ in exactly $k$ blocks $B_1, B_2,\ldots B_k$ with $\# B_i=2n_i$. In the following Lemma we will prove some properties of these numbers that will permit us to compute them recursively.

\vspace{0.4cm}

\begin{lem} Let $p\geq 1$ and $n_1\geq n_2\geq \ldots \geq n_k$ be such that $\sum_{i=1}^{k}{n_i}=p$. Let $\gamma(p\,;n_1,\ldots,n_k)$ be as before, then
\vspace{0.3cm}
\begin{enumerate}
\item $\gamma(p\,;1,1,\ldots,1)=(2p-1)(2p-3)\cdots 1$
\vspace{0.3cm}
\item $\gamma(p\,;p)=1$
\vspace{0.3cm}
\item If $n_1>n_2$ then $\gamma(p\,;n_1,\ldots,n_k)={{2p} \choose {2n_1}} \cdot \gamma(p-n_1\,;n_2,\ldots,n_k)$
\vspace{0.3cm}
\item If exists $r<k$ such that $n_1=n_2=\ldots=n_{r}$ and $n_1>n_{r+1}$ then $$\gamma(p\,;n_1,n_2,\ldots,n_k)=\frac{1}{r!} {2p \choose 2n_1}\ldots
{2p-2(r-1)n_1 \choose 2n_1}\cdot\gamma(p-rn_1\,;n_{r+1},\ldots,n_k).$$
\end{enumerate}
\end{lem}

\begin{proof}
(1) Each element in $\{1,2,\ldots,2p\}$ has to be paired with another. For the first element we have $(2p-1)$ possibilities. Now we remove these two elements and we have $2p-2$ remaining. Each remaining element has to be paired with another, having $(2p-3)$ possibilities. Continuing with this process we get (1). (2) is trivial. (3) In this case, we have only one block of size $2n_1$ and we have exactly ${2p \choose 2n_1}$ possible different blocks like this. We remove this block and we have $2p-2n_1$ elements and we continue with our partition process to get (3). (4) is similar to (3).
\end{proof}

\vspace{0.5cm}
\noindent Given $p \geq 0$ we have that 
\begin{equation*}
X^{2p}=\sum_{1\leq i_1,j_1,\ldots,i_p,j_p}{c_{i_1}c_{j_1}\ldots c_{i_p}c_{j_p}\, R_{i_1}R_{j_1}\ldots R_{i_p}R_{j_p}}
\end{equation*}
and using that $R^{2}=1$ and $\tau(R)=0$ it is not difficult to see that
\begin{equation}\label{eqX}
\tau(X^{2p})=\sum_{k=1}^{p}\Bigg\{ {\sum_{(n_1,n_2,\ldots,n_k)}{\Bigg (\gamma(p\,;n_1,n_2,\ldots,n_k)\cdot \sum_{(p_1,p_2,\ldots,p_k)}{c_{p_1}^{2n_1}c_{p_2}^{2n_2}\ldots c_{p_k}^{2n_k}} \Bigg )}} \Bigg\}
\end{equation}
where the second sum runs over $n_1\geq n_2\geq \ldots \geq n_k$ such that $\sum_{i=1}^{k}{n_i}=p$ and the last one over all the possible $1\leq p_1,\ldots,p_k <+\infty$ with $p_i\neq p_j$ if $i\neq j$.\\
\noindent Analogously,
\vspace{0.3cm}
\begin{equation*}
Y^{2p}=\sum_{1\leq i_1,j_1,\ldots,i_p,j_p}{c_{i_1}c_{j_1}\ldots c_{i_p}c_{j_p}\, T_{i_1}T_{j_1}\ldots T_{i_p}T_{j_p}}
\end{equation*}
and using that $\tau(T^{2n})=(-1)^n$ we get
\begin{equation*}
\tau(Y^{2p})=(-1)^p \sum_{k=1}^{p}\Bigg\{ {\sum_{(n_1,n_2,\ldots,n_k)}{\Bigg (\gamma(p\,;n_1,n_2,\ldots,n_k)\cdot \sum_{(p_1,p_2,\ldots,p_k)}{c_{p_1}^{2n_1}c_{p_2}^{2n_2}\ldots c_{p_k}^{2n_k}} \Bigg )}} \Bigg\}.
\end{equation*}

\vspace{0.3cm}

\noindent Hence, $\tau(X^{2p+1})=\tau(Y^{2p+1})=0$ and $\tau(Y^{2p})=(-1)^p \tau(X^{2p})$. Since $\mathrm{Re}(A)=\frac{1}{2}X$ and $\mathrm{Im}(A)=\frac{1}{2i}Y$ we see that $\mathrm{Re}(A)$ and $\mathrm{Im}(A)$ have the same distribution. More precisely, we can state the following Proposition.

\vspace{0.5cm}

\begin{prop}
Let $A$ be as before and let $a=\mathrm{Re}(A)$ and $b=\mathrm{Im}(A)$. Then
\begin{displaymath}
\tau(a^n)=\tau(b^n)= \left\{ \begin{array}{ll}
0 & \textrm{\,if\,\,\,\,} n=2p+1\\
(\frac{1}{2})^{2p}\tau(X^{2p}) & \textrm{\,if\,\,\,\,} n=2p
\end{array} \right.
\end{displaymath}
where $\tau(X^{2p})$ is as in equation (\ref{eqX}).
\end{prop}

\vspace{0.5cm}

\noindent Another way of looking at the operator $X$,
\begin{equation}
X=\sum_{n=1}^{+\infty}{c_n\,R_n}  \mathrm{\,\,\,\,\,\,where\,\,\,\,\,\,} R_n=I^{\otimes(n-1)}\otimes R \mathrm{\,\,\,\,with\,\,\,\,} 
R=\left( \begin{array}{cc}
0 & 1\\
1 & 0\\
\end{array} \right),
\end{equation}
is as a measurable function in $[-1,1]$. The operators $\{R_n\}_n$ are selfadjoint and commute with each other. Therefore, we can think them as independent random variables in $[-1,1]$. Moreover, if we think $R_n$ as a function $f_n$ in $[-1,1]$ then these functions satisfy that $\tau(R_n)=\int_{-1}^{1}f_n(x) dx=0$ and $f_n^2(x)=1$ and we can picture them as,

\begin{figure}[Ht]
\begin{center}
\includegraphics[width=5cm]{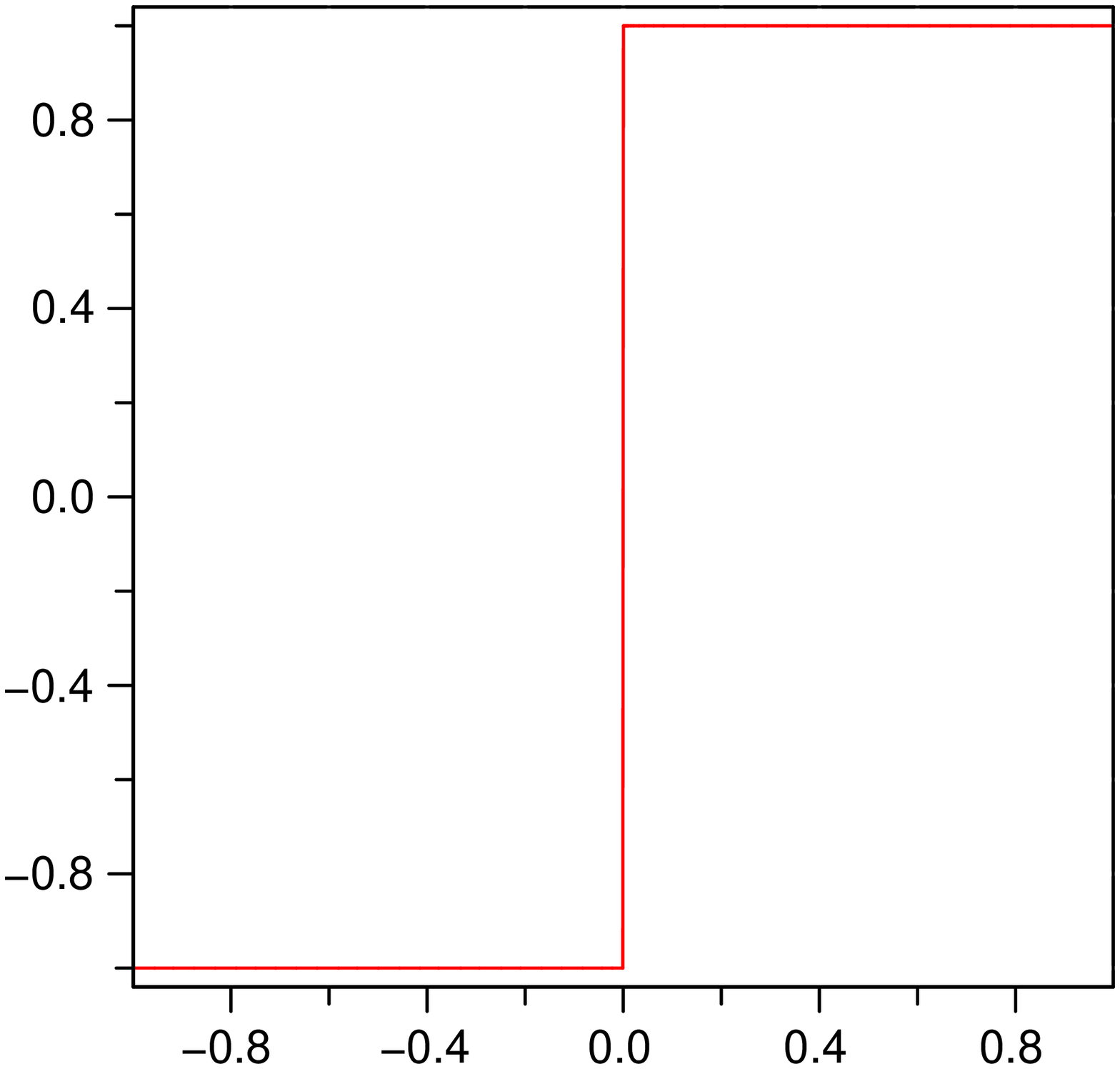}
\includegraphics[width=5cm]{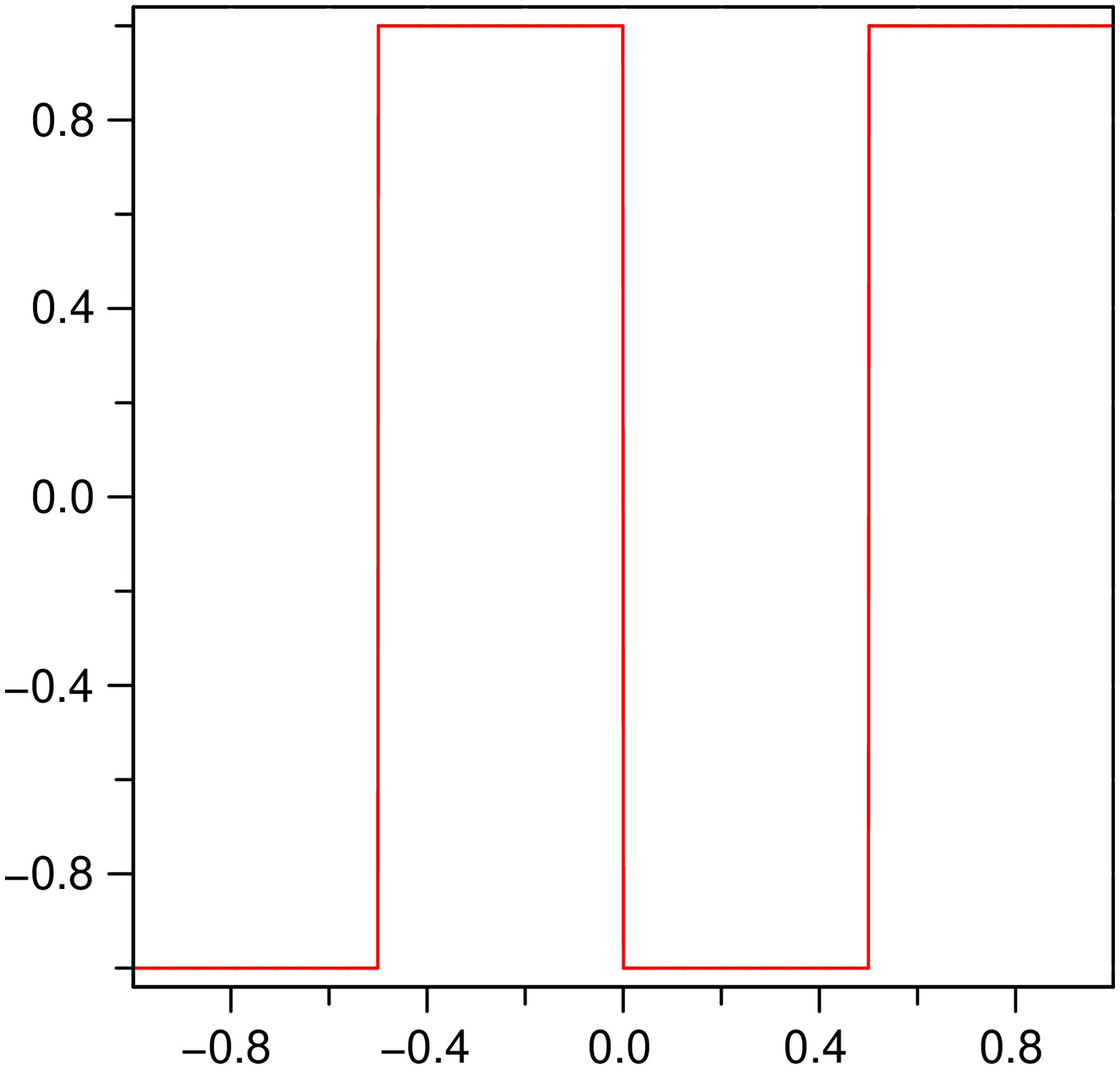}
\includegraphics[width=5cm]{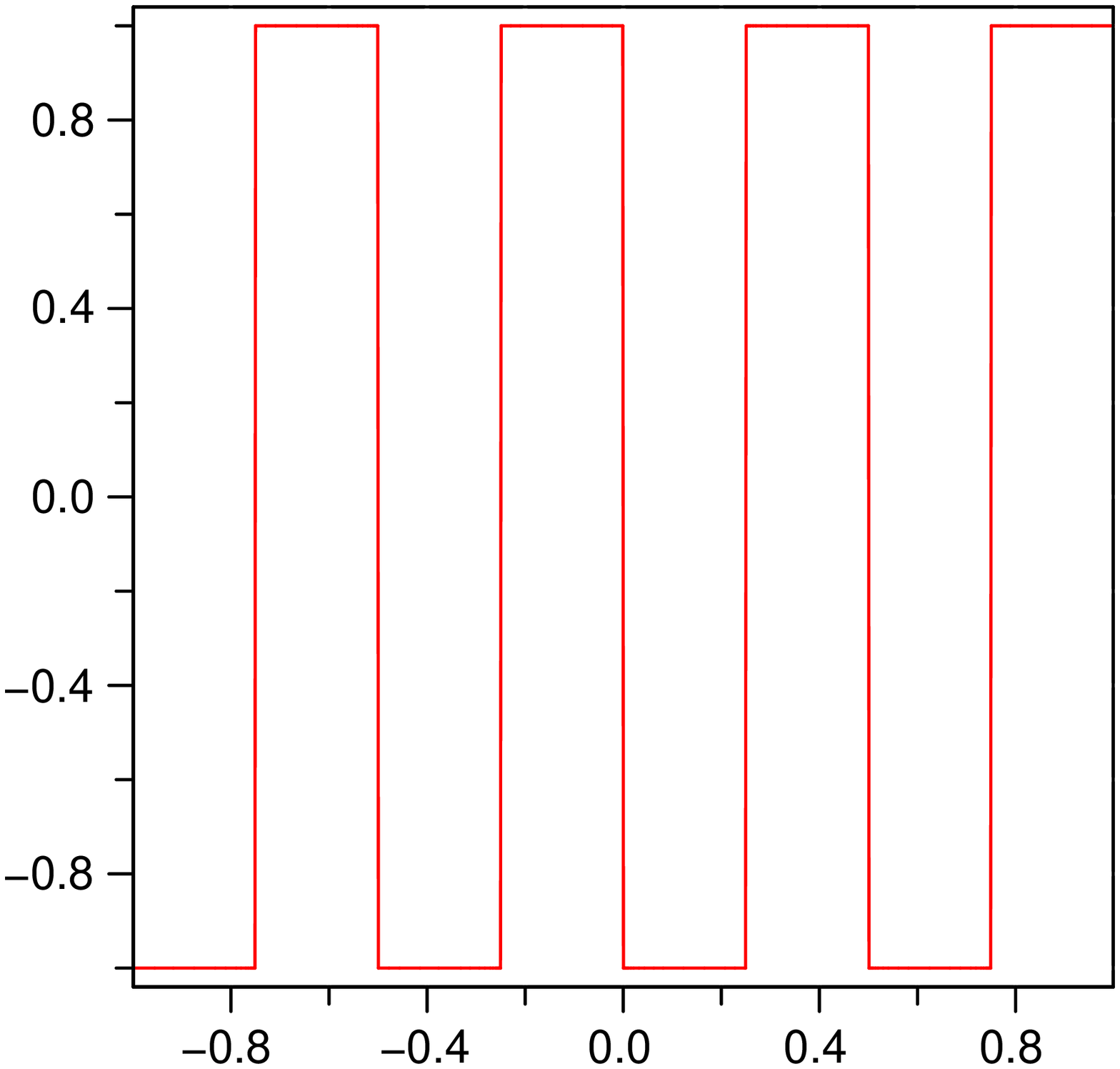}
\caption{Functions $f_1(x)$, $f_{2}(x)$ and $f_3(x)$}
\end{center}
\end{figure}

\vspace{0.4cm}
\begin{displaymath}
f_1(x):=\left\{ \begin{array}{ll}
-1  & \textrm{if} \quad -1<x<0\\
 1  & \textrm{if} \quad\quad\  0<x<1\\
 0  & \textrm{otherwise}
\end{array} \right.
\end{displaymath}

\vspace{0.4cm}

\noindent and $f_{2}(x)=f_1(2x+1)+f_1(2x-1)$ and in general, $$f_{n+1}(x)=f_n(2x+1)+f_n(2x-1) \textrm{\quad for $x\in[-1,1]$ and $n\geq 1$}.$$

\vspace{0.2cm}

\noindent Hence, we can represent the operator $X$ by the measurable function $f(x)$, given by $f(x)=\sum_{n=1}^{+\infty}{c_n\,f_n(x)}$ and  
\begin{equation}
\tau(X^{n})=\int_{-1}^{1}{f(x)^{n}\,dx} \textrm{\,\,\,\,\,\,for all \,\,} n\geq 0.
\end{equation}

\vspace{0.4cm}

\noindent Note that in the case $c_{n}=(\frac{1}{2})^n$ we get that $f(x)=x$ on $[-1,1]$. Hence, $\tau(X^{2p})=\int_{-1}^{1}{x^{2p}\,dx}=\frac{2}{2p+1}$ and therefore
\begin{displaymath}
\tau(a^n)=\tau(b^n)= \left\{ \begin{array}{ll}
0 & \textrm{\,if\,\, $n$ odd}\\
(\frac{1}{2})^{n}\frac{2}{n+1} & \textrm{\,if\,\, $n$ even}.
\end{array} \right.
\end{displaymath}

\vspace{0.4cm}

\begin{figure}[Ht]\label{hfig2}
\begin{center}
\includegraphics[width=7cm]{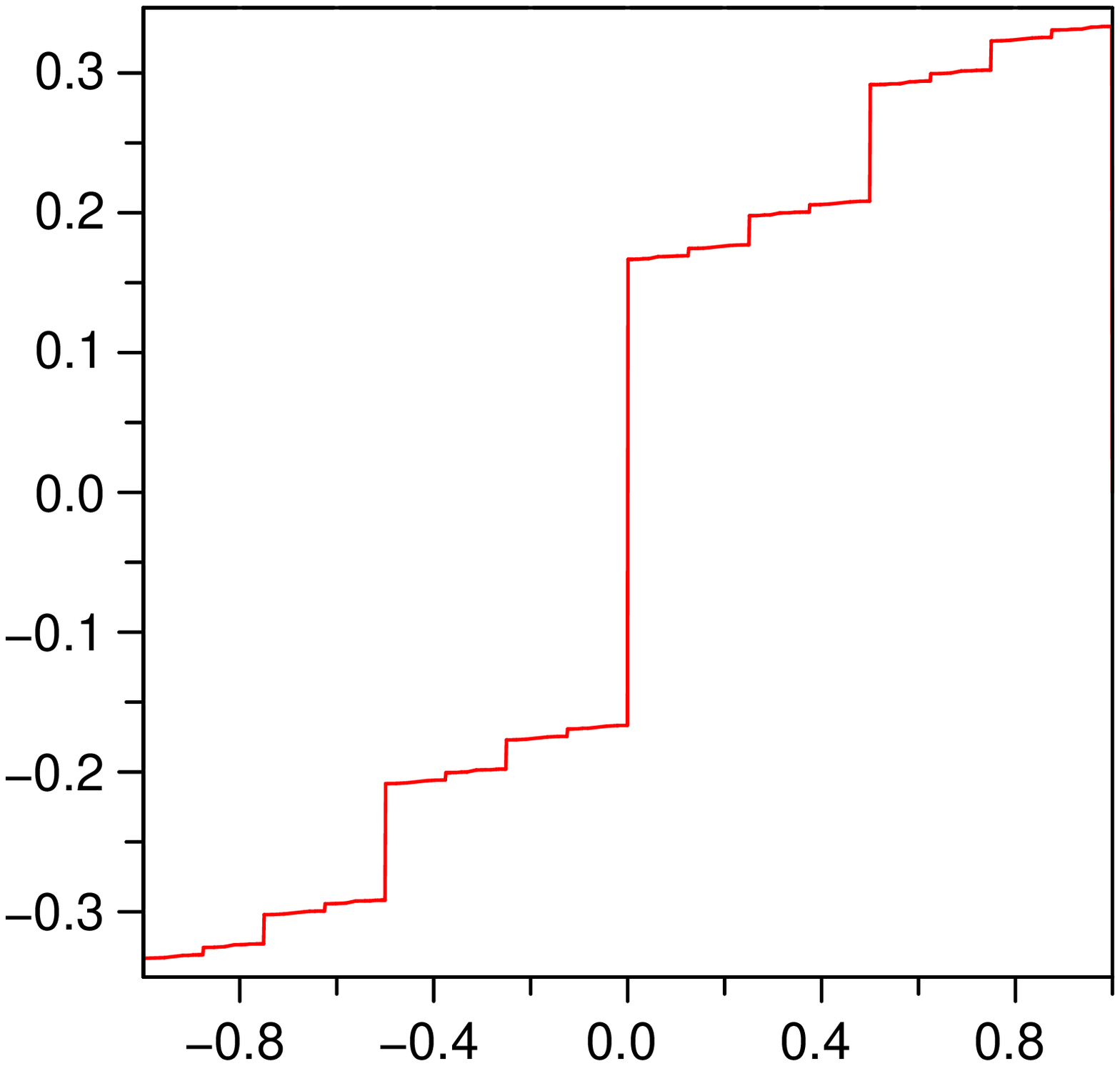}
\caption{Function $\sum_{k=1}^{10}{\Big(\frac{1}{4}\Big)^{k}f_{k}(x)}$}
\end{center}
\end{figure}

\begin{obs}
The spectrum of the operator $X$, $\sigma(X)$, is the image of the function $f$. Therefore, if $c_1>\sum_{n=2}^{+\infty}{c_n}$ then
$0\notin \sigma(X)$ and the operator $X$ is invertible and so $a$ and $b$. This is the case, for example, of $c_{n}=\alpha^{n}$ when $0<\alpha<\frac{1}{2}$ (see Figure 2). Note also that $\sigma(a)=\sigma(b)\subseteq [-s/2,s/2]$ where $s=\sum_{n=1}^{+\infty}{c_n}$.
\end{obs}

\vspace{0.4cm}

In probability theory, the characteristic function of any random variable completely defines its probability distribution. On the real line it is given by the following formula, where $Z$ is any random variable with the distribution in question:
$$\varphi_{Z}(t):=\mathbb{E}(e^{itZ})$$
where $t$ is a real number, $i$ is the imaginary unit, and $\mathbb{E}$ denotes the expected value.
Characteristic functions are particularly useful for dealing with functions of independent random variables. In particular, if $Z_{1}$ and $Z_{2}$ are independent random variables then $\varphi_{Z_1+Z_2}(t)=\varphi_{Z_1}(t)\varphi_{Z_2}(t)$. Characteristic functions can also be used to find moments of random variables. Provided that $n$-th moment exists, characteristic function can be differentiated $n$ times and the following formula holds
\begin{equation}\label{moments}
\mathbb{E}(Z^{n})=\Big(\frac{1}{i}\Big)^{n}\Bigg[\frac{d^{\,n}}{dt^n}\varphi_{Z}(t)\Bigg]_{t=0}.
\end{equation}
We will compute the characteristic function of $X_{\alpha}:=\sum_{n=1}^{+\infty}{\alpha^{n}R_{n}}$. For each $n\geq 1$, $\tau(R_{n}^{k})=0$ if $k$ is odd and $1$ if $k$ is even. Therefore, 
$$\varphi_{\alpha^{n}R_{n}}(t)=\sum_{k=0}^{+\infty}{(-1)^{k}\frac{t^{2k}}{(2k)!}\alpha^{2nk}}=\cos(\alpha^{n}t)$$
hence
$$\varphi_{X_{\alpha}}(t)=\prod_{n=1}^{+\infty}{\cos(\alpha^{n}t)}.$$
Then we can use (\ref{moments}) and the last equation to compute the even moments of $X_{\alpha}$. For example, using this formula we can see that $\tau(X_{\alpha}^2)=\frac{\alpha^{2}}{1-\alpha^{2}}$.

\vspace{0.3cm}
\begin{prop}
The operators $a$ and $b$ clearly do not commute but $\tau(a^nb^m)=\tau(a^n)\tau(b^m)=\tau(a^n)\tau(a^m)$ for all $n$ and $m$.
\end{prop}

\begin{proof}
The last equality is trivial since $a$ and $b$ have the same distribution. To prove the first equality it is enough to prove that $\tau(X^{n}Y^{m})=\tau(X^{n})\tau(Y^{m})$ for all $n$ and $m$. Since, 
\begin{equation*}
X^{n}=\sum_{1\leq l_1,\ldots,l_n}{c_{l_1}\ldots c_{l_n}\, R_{l_1}\ldots R_{l_n}} \textrm{\,\,\,\,\,and \,\,\,}
Y^{m}=\sum_{1\leq k_1,\ldots,k_m}{c_{k_1}\ldots c_{k_m}\, T_{k_1}\ldots T_{k_m}},
\end{equation*}
to prove $\tau(X^{n}Y^{m})=\tau(X^{n})\tau(Y^{m})$ it is enough to prove that $$\tau(R_{l_1}\ldots R_{l_n}T_{k_1}\ldots T_{k_m})=\tau(R_{l_1}\ldots R_{l_n})\tau(T_{k_1}\ldots T_{k_m})$$ and this is true since $\tau(R^{l}T^{h})=\tau(R^{l})\tau(T^{h})$ for all $l$ and $h$.
\end{proof}

\vspace{0.6cm}

The family of operators $\{R_{n}\}_{n=1}^{+\infty}$ is a commuting family of selfadjoint operators. If we denote by 
$$\N_{2}(\C):=\Bigg \{\left( \begin{array}{cc} \alpha & \beta\\ \beta & \alpha\\ \end{array} \right)\,:\,\alpha, \beta\in\ \C \Bigg \}\subset M_{2}(\C),$$ then it is not difficult to see that
\begin{equation}\label{cartan}
\mathcal{A}:=W^{*}(\{R_{n}\}_{n=1}^{+\infty})=\overline{\Bigg (\bigotimes_{n=1}^{+\infty}{\N_{2}(\C)}\Bigg )}^{\mathrm{WOT}}
\end{equation}
which is a Cartan masa in the hyperfinite $\mathrm{II}_1$--factor $\mathcal{R}$. It is clear that $W^{*}(a)=W^{*}(X)\subseteq \mathcal{A}$. A natural question is when is $W^{*}(a)=\mathcal{A}$? Is $W^{*}(a)$ always a diffuse abelia subalgebra of $\mathcal{A}$? 

\vspace{0.4cm}

\begin{obs}\label{masa}
Consider the projections 
$$p_{n}:=\frac{1}{2}\,\,I^{\otimes(n-1)}\otimes \left( \begin{array}{cc} 1 & 1\\ 1 & 1\\ \end{array} \right) \quad \text{and}\quad q_{n}:=\frac{1}{2}\,\,I^{\otimes(n-1)}\otimes \left( \begin{array}{cc} 1 & -1\\ -1 & 1\\ \end{array} \right).$$
Then $R_{n}=p_{n}-q_{n}$ and $X=\sum_{n=1}^{+\infty}{c_{n}R_{n}}$. If for all $n\geq 1$, $c_{n}\geq \sum_{k=n+1}^{+\infty}{c_{k}}$ then the function $f$ is increasing and we can recover $p_{n}$ and $q_{n}$ as spectral projections of $X$ and hence $W^{*}(X)=\mathcal{A}$. This is the case, for example, of $c_{n}=\alpha^{n}$ when $0<\alpha\leq \frac{1}{2}$.
\end{obs}

\vspace{0.3cm}
\noindent The following Theorem answers the questions asked before.

\vspace{0.2cm}

\begin{teo}
Let $0<\alpha<1$ and $X_{\alpha}=\sum_{n=1}^{+\infty}{\alpha^{n}R_{n}}$. Then the abelian algebra $W^{*}(X_{\alpha})$ is always diffuse. If $\,\,0<\alpha\leq \frac{1}{2}$ then $W^{*}(X_{\alpha})$ is the Cartan masa $\mathcal{A}$ as in (\ref{cartan}). However, if there exist a polynomial $p(x)=a_{1}x^{n_1}+a_{2}x^{n_2}+\ldots+a_{k}x^{n_k}$ with coefficients $a_{i}\in\{1,-1\}$ such that $p(\alpha)=0$ (for example $\alpha=\frac{\sqrt{5}-1}{2}$) then $W^{*}(X_{\alpha})\subsetneq \mathcal{A}$. 
\end{teo}

\vspace{0.2cm}

\begin{proof}
The case $0<\alpha\leq \frac{1}{2}$ was discussed in Remark \ref{masa}.\\

\vspace{0.1cm}
\noindent Consider the Bernoulli space $(M,\mu)=\Big(\prod_{n=1}^{+\infty}{\{1,-1\}},(\frac{1}{2}(\delta_{1}+\delta_{-1}))^{\otimes \N}\Big)$.
We can model the selfadjoint element $X_{\alpha}$ as the measurable function $g_{\alpha}:(M,\mu)\to \R$ defined by $g_{\alpha}(\{\epsilon_n\})=\sum_{n=1}^{+\infty}{\epsilon_{n}\alpha^{n}}$. In order to prove that $W^{*}(X_{\alpha})$ is diffuse it is equivalent to prove that $W^{*}(g_{\alpha})\subseteq L^{\infty}(M,\mu)$ is diffuse. Assume this is not true, hence there exists $\beta\in\R$ such that $\mu(g_{\alpha}^{-1}(\{\beta\}))=\gamma>0$. Denote by $E$ the set $E:=g_{\alpha}^{-1}(\{\beta\})$. For each $n\geq 1$, we define 

$$E_{n}^{+}:=\{x=\{\epsilon_k\}_{k}\in E\,:\,\epsilon_{n}=1\}\quad \text{and}\quad E_{n}^{-}:=\{x=\{\epsilon_k\}_{k}\in E\,:\,\epsilon_{n}=-1\}.$$
\vspace{0.1cm}

\noindent It is clear that for each $n\geq 1$ the sets $E_{n}^{+}$ and $E_{n}^{+}$ are measurable sets, $E_{n}^{+}\cup E_{n}^{-}=E$ and $E_{n}^{+}\cap E_{n}^{-}=\emptyset$. Hence, for each $n$, either $E_{n}^{+}$ or $E_{n}^{-}$ has measure bigger or equal than $\gamma/2$. If $\mu(E_{n}^{+})> \gamma/2$ then define $F_{n}:=\{(\epsilon_1,\ldots,\epsilon_{n-1},-1,\epsilon_{n+1},\ldots)\,:\,\{\epsilon_{k}\}\in E_{n}^{+}\}$ and if $\mu(E_{n}^{-})\geq \gamma/2$ then define $F_{n}:=\{(\epsilon_1,\ldots,\epsilon_{n-1},1,\epsilon_{n+1},\ldots)\,:\,\{\epsilon_{k}\}\in E_{n}^{-}\}$.
By definition, $\mu(F_{n})\geq\gamma/2$ and if $x\in F_{n}$ then $g_{\alpha}(x)$ is either $\beta+2\alpha^{n}$ or $\beta-2\alpha^{n}$. Assume there exists $x\in F_{n}\cap F_{m}$ then $\beta\pm 2\alpha^{m}=\beta\pm 2\alpha^{n}$ and hence $\alpha^{n}=\pm \alpha^{m}$ then $n=m$. Therefore, we constructed a sequence of disjoint measurable sets $\{F_{n}\}_{n}$ each of measure $\mu(F_{n})\geq \gamma/2$ which is clearly impossible. Therefore, $W^{*}(X_{\alpha})$ is diffuse.\\

\vspace{0.1cm}
\noindent Let $p(x)$ be a polynomial $p(x)=a_{1}x^{n_1}+a_{2}x^{n_2}+\ldots+a_{k}x^{n_{k}}$ with coefficients $a_{i}\in\{1,-1\}$ and $\alpha\in (0,1)$ be such that $p(\alpha)=0$. (Note that there are infinitely many countable $\alpha$ in $(\frac{1}{2},1)$ with this property but none in $(0,\frac{1}{2}]$). Define the cylindrical sets 

$$G_1:=\{\{\epsilon_{n}\}_{n}\,:\,\epsilon_{n_i}=a_{i},\, i=1,\ldots,k\}\,\,\text{and}\,\, G_2:=\{\{\epsilon_{n}\}_{n}\,:\,\epsilon_{n_i}=-a_{i},\, i=1,\ldots,k\}$$ 
\vspace{0.1cm}

\noindent it is clear that $G_{1}\cap G_{2}=\emptyset$ and that $\mu(G_{1})=\mu(G_{2})=\frac{1}{2^k}$. The function $g_{\alpha}$ does not separates this two cylindrical sets and hence $W^{*}(g_{\alpha})\neq L^{\infty}(M,\mu)$. Therefore, $W^{*}(X_{\alpha})\neq \mathcal{A}$.
\end{proof}

\vspace{0.6cm}

\section{Moments of $A^{*}A$}

In this section we will give a combinatorial formula describing the moments of $A^{*}A$. Let 
$\{c_n\}_n\in l_1(\N)$ and $A=\sum_{n=1}^{+\infty}{c_n\,V_n}$ where $V_n=I^{\otimes(n-1)}\otimes V$. Then given $p\geq 1$ we see that
$$(A^*A)^p=\sum_{1\leq n_1,m_1,\ldots,n_p,m_p}{\overline{c_{n_1}}c_{m_1}\overline{c_{n_2}}c_{m_2}\ldots\overline{c_{n_p}}c_{m_p}\,
\,V_{n_1}^{*}V_{m_1}V_{n_2}^{*}V_{m_2}\ldots V_{n_p}^{*}V_{m_p}}.$$

\noindent For $p\geq 1$ consider $p$ elements of color red and $p$ of color white. Order them linearly and alternating the colors. Let $1\leq k\leq p$ and $n_1\geq n_2\geq \ldots \geq n_k$ be such that $\sum_{i=1}^{k}{n_i}=p$. We define $\alpha(p\,;n_1,\ldots,n_k)$ the number of partitions of these $2p$ elements in $k$ blocks $B_1, B_2, \ldots, B_k$ of size $2n_1, 2n_2, \ldots, 2n_k$ such that each block contains the same amount of element of each color and are alternating, i.e.: if we look at the elements of one block the colors are alternating. 

\vspace{0.2cm}
\begin{ej} For the case $p=2$ we have that $\alpha(2\,;1,1)=2$, $\alpha(2\,;2)=1$. For $p=3$ we have $\alpha(3\,;1,1,1)=6$, 
$\alpha(3\,;2,1)=6$ and $\alpha(3\,;3)=1$. Some of the possibles partitions for $p=3$ can be seen in Figure 3.
\end{ej}
\vspace{0.2cm}

\begin{figure}[Ht]
\begin{center}
\includegraphics[width=5cm]{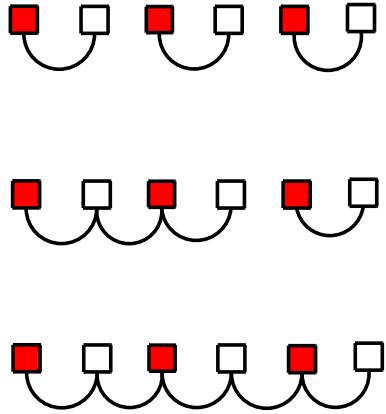}
\caption{An element of $\alpha(3\,;1,1,1)$ an element of $\alpha(3\,;2,1)$ and the only element of $\alpha(3\,;3)$}
\end{center}
\end{figure}

Given $1\leq n\leq p$ let us denote $\beta(p\,;n)$ the number of blocks of size $2n$ satisfying the alternating condition. We first choose the $n$ elements of red color which will be located at the positions 
$1\leq r_1<r_2<\ldots <r_n\leq 2p-1$ (note that the red elements are located at the odd integers while the white at the even). Then we choose the $n$ elements of white color. In order to satisfy the alternating condition, the positions $\{w_i\}_{i=1}^{n}$ of the white elements have to satisfy either
$$1\leq r_1<w_1<r_2< w_2 <\ldots <r_n< w_n\leq 2p$$
or
$$1\leq w_1<r_1<w_2<r_2 <\ldots <w_n<r_n\leq 2p-1.$$

\noindent If $r_1=1$ then we have $$\frac{1}{2^n}(r_2-1)\cdot(r_3-r_2)\ldots (r_n-r_{n-1})\cdot(2p+1-r_n)$$ possibilities to choose the white elements. If $r_1>1$ we have the option of either start with white or with red. Starting with white we have $$\frac{1}{2^n}(r_1-1)(r_2-r_1)\cdot(r_3-r_2)\ldots (r_n-r_{n-1})$$ and starting with red we have 
$$\frac{1}{2^n}(r_2-r_1)\cdot(r_3-r_2)\ldots (r_n-r_{n-1})\cdot(2p+1-r_n).$$
Then
\begin{eqnarray*}
\beta(p\,;n) & = & \frac{1}{2^n}\Bigg (\sum_{1\leq r_1<\ldots <r_n}{(r_2-r_1)\cdot(r_3-r_2)\ldots (r_n-r_{n-1})\cdot(2p+1-r_n)}\\ 
& + & \sum_{2\leq r_1<\ldots <r_n}{(r_1-1)(r_2-r_1)\cdot(r_3-r_2)\ldots (r_n-r_{n-1})} \Bigg).
\end{eqnarray*}

\vspace{0.3cm}
\noindent Note that $\beta(p\,;p-1)=2{p \choose p-1}.$

\vspace{0.6cm}

\noindent The next Lemma provides us with some information about the combinatorial numbers $\alpha(p\,;n_1,\ldots,n_k)$.

\vspace{0.2cm}
\begin{lem} Let $p\geq 1$ and $n_1\geq n_2\geq \ldots \geq n_k$ such that $\sum_{i=1}^{k}{n_i}=p$ let $\alpha(p\,;n_1,\ldots,n_k)$ then
\begin{enumerate}
\vspace{0.4cm}
\item $\alpha(p\,;1,1,\ldots,1)=p!$
\vspace{0.3cm}
\item $\alpha(p\,;p)=1$
\vspace{0.3cm}
\item $\alpha(p\,;p-n,1,\ldots,1)=\beta(p\,;p-n)\cdot n!$
%\item If $n_1>n_2$ then $\alpha(p\,;n_1,\ldots,n_k)=\beta(p\,;n_1) \cdot \alpha(p-n_1\,;n_2,\ldots,n_k)$ 
%\item If $n_1$ appears $r$ times then $$\alpha(p\,;n_1,\ldots,n_1,n_2,\ldots,n_k)=\frac{1}{r!}\cdot\beta(p\,;n_1)
%\ldots \beta(p-(r-1)n_1\,;n_1) \cdot \alpha(p-rn_1\,;n_2,\ldots,n_k)$$
\end{enumerate}
\end{lem}

\vspace{0.5cm}

\begin{prop}\label{m} Let $p\geq 1$ then 
\begin{equation*}
\tau((A^*A)^p)=\sum_{k=1}^{p}{\Bigg[\frac{1}{2^k} \cdot \sum_{(n_1,\ldots,n_k)}{\Bigg( \alpha(p\,;n_1,\ldots,n_k)\cdot \sum_{ p_1,\ldots,p_k}{|c_{p_1}|^{2n_1}|c_{p_2}|^{2n_2}\ldots |c_{p_k}|^{2n_k}}\Bigg) } \Bigg]}
\end{equation*}
where the second sum runs over all the $k$--tuples such that $n_1\geq n_2\geq \ldots \geq n_k$ with $\sum_{i=1}^{k}{n_i}=p$, and the last one over all the possible $1\leq p_1,\ldots,p_k <+\infty$ with $p_i\neq p_j$ if $\,i\neq j$.\\
\end{prop}

\begin{proof}
We have to observe that
$$(A^*A)^p=\sum_{1\leq n_1,m_1,\ldots,n_p,m_p}{\overline{c_{n_1}}c_{m_1}\overline{c_{n_2}}c_{m_2}\ldots\overline{c_{n_p}}c_{m_p}\,
\,V_{n_1}^{*}V_{m_1}V_{n_2}^{*}V_{m_2}\ldots V_{n_p}^{*}V_{m_p}}$$
and since $\tau(V)=\tau(V^{*})=0$, $V^{2}={V^{*}}^{2}=0$, $VV^{*}=P$ and $V^{*}V=Q$ then $\tau(V_{n_1}^{*}V_{m_1}V_{n_2}^{*}V_{m_2}\ldots V_{n_p}^{*}V_{m_p})$ is going to be nonzero if all the $V$'s are paired with $V^{*}$'s in an alternating way. Using the definition of the numbers 
$\alpha(p\,;n_1,\ldots,n_k)$ it is not difficult to see that the formula in the Proposition follows.

\end{proof}

\vspace{0.8cm}

\begin{question}
Is there a nice formula, or recursive description of the numbers \\$\alpha(p\,;n_1,\ldots,n_k)$? If we fix $k$, can we at least compute recursively 
$$s_{p}(k):=\sum_{(n_1,n_2,\ldots,n_k)}{\alpha(p\,;n_1,\ldots,n_k)} \textrm{\quad where \quad} n_1\geq\ldots\geq n_k \textrm{\quad with\quad} \sum_{i=1}^{k}{n_i}=p\,\,?$$
\end{question}

\vspace{0.7cm}

\begin{center}
\begin{tabular}{|l|l|l|}
\hline
$p$ & $\alpha(p\,;n_1,\ldots,n_k)$ & $s_{p}(k)$\\
\hline \hline
$p=1$ & $\alpha(1;1)=1$ & $s_{1}(1)=1$\\
\hline
$p=2$ & $\alpha(2;2)=1$ & $s_{2}(1)=1$\\
      & $\alpha(2;1,1)=2$ & $s_{2}(2)=2$\\
\hline
$p=3$ & $\alpha(3;3)=1$ & $s_{3}(1)=1$\\
      & $\alpha(3;2,1)=6$ & $s_{3}(2)=6$\\
      & $\alpha(3;1,1,1)=6$ & $s_{3}(3)=6$\\
\hline
$p=4$ & $\alpha(4;4)=1$ & $s_{4}(1)=1$\\
      & $\alpha(4;3,1)=8$ & $s_{4}(2)=14$\\
      & $\alpha(4;2,2)=6$ & $s_{4}(3)=40$\\ 
      & $\alpha(4;2,1,1)=40$ & $s_{4}(4)=24$\\
      & $\alpha(4;1,1,1,1)=24$ & \\
\hline
\end{tabular}
\end{center}

\begin{figure}[Ht]
\begin{center}
\includegraphics[width=6cm]{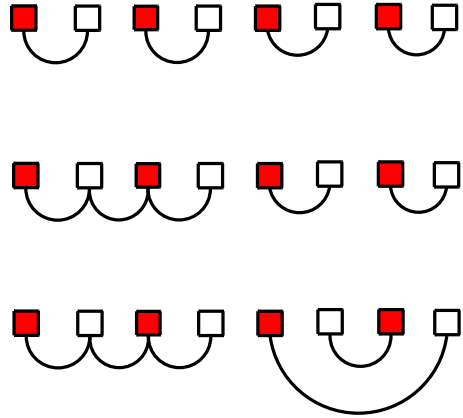}
\caption{An element of $\alpha(4\,;1,1,1,1)$ an two elements of $\alpha(4\,;2,1,1)$}
\end{center}
\end{figure}

\end{document}